\documentclass[12pt]{amsart}

\usepackage[utf8x]{inputenc}
\usepackage{tikz-cd}

\usepackage{amsmath, amsthm, amsfonts, amssymb}
\usepackage{epsfig, enumerate}
\usepackage{color}
\usepackage[left=3.5cm, marginratio=1:1]{geometry}

\usepackage{hyperref}
\usepackage[all]{xy}

\usepackage{float}

\usepackage{tikz}
\usetikzlibrary{arrows,shapes,trees}
\usepackage{graphicx}
\usepackage{amssymb}
\usepackage{amsfonts}
\usepackage{amsthm}
\usepackage{amsmath}
\usepackage{mathrsfs}

\newtheorem{maintheorem}{Theorem} 

\newtheorem{theorem}{Theorem}
\newtheorem{lemma}[theorem]{Lemma}
\newtheorem{coro}[theorem]{Corollary}
\newtheorem{prop}[theorem]{Proposition}

\newtheorem{conjecture}{Conjecture}

\newtheorem{problem}{Problem}

\newtheorem*{assumptions*}{Assumptions}

\newtheorem*{rem*}{Remark}

\theoremstyle{remark}
\newtheorem{remark}[theorem]{Remark}
\newtheorem*{remark*}{Remark}

\theoremstyle{example}
\newtheorem{example}[theorem]{Example}

\theoremstyle{definition}
\newtheorem{definition}{Definition}

\newcommand{\C}{{\mathbf C}}

\newcommand{\M}{{\mathbf M}}

\newcommand{\T}{{\mathbf T}}

\newcommand{\Z}{{\mathbf Z}}

\newcommand{\Aa}{{\mathcal A}}

\newcommand{\Ee}{{\mathcal E}}

\newcommand{\Gg}{{\mathcal G}}

\newcommand{\Kk}{{\mathcal K}}
\newcommand{\Ll}{{\mathcal L}}

\newcommand{\Uu}{{\mathcal U}}
\newcommand{\Vv}{{\mathcal V}}
\newcommand{\Ww}{{\mathcal W}}

\newcommand{\cZ}{{\mathcal Z}}

\newcommand{\CC}{{\mathbb C}}

\newcommand{\NN}{{\mathbb N}}

\newcommand{\QQ}{{\mathbb Q}}
\newcommand{\RR}{{\mathbb R}}

\newcommand{\TT}{{\mathbb T}}

\newcommand{\ZZ}{{\mathbb Z}}

\newcommand{\SL}{{\rm SL}}
\newcommand{\GL}{{\rm GL}}

\newcommand{\interior}{{\operatorname{int}}}

\DeclareMathOperator{\supp}{{\rm supp}}

\DeclareMathOperator{\Aff}{\operatorname{Aff}}
\DeclareMathOperator{\Lip}{\operatorname{Lip}}

\newcommand{\be}[1]{\begin{equation} \label{#1} }
\newcommand{\ee}{\end{equation}}
\newcommand{\beq}{\begin{equation}}

\def \Zk{{\mathbb Z^k}}
\def \Diff{{\rm Diff}}

\def \Lip{{\rm Lip}}

\def \al{{\alpha}}
\def \bal{{\bar{\alpha}}}
\def \Rk {{\mathbb R^k}}

\def \C{{\mathscr{C}}}

\def \Ad{{\mathrm{Ad}}}

\def \bal{\bar{\alpha}}

\def \hx0{\hat{x_0}}

\def \Hom{\mathrm{Homeo}}

\def \id{\mathrm{id}}
\def \Id{\operatorname{Id}}
\def \Haar{\operatorname{Haar}}

\def \Z{\mathcal{Z}}

\begin{document}

\title[The symmetries of affine $K$-systems]{The symmetries of affine $K$-systems and a program for centralizer rigidity}

\author{Danijela Damjanovi\'c}

\address[Damjanovi\'c]{Department of mathematics, Kungliga Tekniska hogskolan, Lindstedtsv\"agen 25, SE-100 44 Stockholm, Sweden.} 
\email{ddam@kth.se}

\author{Amie Wilkinson}
\address[Wilkinson]{Department of mathematics, the University of Chicago, Chicago, IL, US, 60637}
\email{amie.wilkinson.math@gmail.com}

\author{Chengyang  Wu}
\address[Wu]{School of mathematical sciences, Peking University, Beijing, China, 100871}
\email{chengyangwu@stu.pku.edu.cn}

\author{Disheng Xu}
\address[Xu]{School of Sciences, Great Bay University and Great Bay Institute for Advanced Study, Songshan Lake Innovation Entrepreneurship Community A5,  Dongguan, Guangdong, China, 523000}
\email{xudisheng@gbu.edu.cn}

\begin{abstract} Let $\Aff(X)$ be the group of affine diffeomorphisms of a closed homogeneous manifold $X=G/B$ admitting  a $G$-invariant Lebesgue-Haar probability measure $\mu$.
For $f_0\in \Aff(X)$, let $\cZ^\infty(f_0)$ be the
 group of $C^\infty$ diffeomorphisms of $X$ commuting with $f_0$. This paper addresses the question: for which $f_0\in \Aff(X)$ is
  $\cZ^\infty(f_0)$ a Lie subgroup of $\Diff^\infty(X)$? 
  Among our main results are the following.
\begin{itemize}
\item If $f_0\in \Aff(X)$ is weakly mixing with respect to $\mu$, then $\cZ^\infty(f_0)<  \Aff(X)$, and hence is a Lie group.
\item If $f_0\in \Aff(X)$ is ergodic with respect to $\mu$, then $\cZ^\infty(f_0)$ is a (necessarily $C^0$ closed) Lie subgroup of $\Diff^\infty(X)$ (although not necessarily a subgroup of $\Aff(X)$).
\item If $f_0\in \Aff(X)$ fails to be a $K$-system with respect to $\mu$, then there exists $f\in \Aff(X)$ arbitrarily close to $f_0$ such that $\cZ^\infty(f)$ is not a Lie group, containing as a continuously embedded subgroup  either the abelian group $C^\infty_c((0,1))$ (under addition) or the simple group $\Diff^\infty_c((0,1))$ (under composition).

\item 
Considering perturbations of $f_0$ by left translations, we conclude that $f_0$ is stably ergodic if and only if the condition
 $\cZ^\infty < \Aff(X)$ holds in a neighborhood of $f_0$ in $\Aff(X)$.
 (Note that by \cite{BS97, Dani77} $f_0\in\Aff(X)$ is stably ergodic in  $\Aff(X)$ if and only if $f_0$ is a $K$-system.)
 
\end{itemize}
 
The affine $K$-systems are precisely those that are partially hyperbolic and essentially accessible, belonging to a class of diffeomorphisms whose dynamics have been extensively studied.  In addition, the properties of partial hyperbolicity and accessibility are stable under $C^1$-small perturbation, and in some contexts, essential accessibility has been shown to be stable under smooth perturbation.

Considering the smooth perturbations of affine $K$-systems, we outline a full program for (local) centralizer rigidity. Special cases of realization of the program are \cite{DWXcent}, \cite{ZhijingWang}, \cite{Sven}, \cite{Sandfeldt2024}, and the forthcoming work \cite{DSWX} and \cite{DWX rank 1 factor}. 
\end{abstract}
\maketitle
\tableofcontents
\section*{Introduction}\label{sec: intro}

The {\em symmetries} of a smooth dynamical system 
 $f\colon M\to M$ on a closed manifold $M$ are the coordinate changes that preserve the dynamics -- that is, the set of diffeomorphisms of $M$ that commute with $f$ under composition. The set of symmetries forms a group under composition, and when $f$ itself is a diffeomorphism, this group is the {\em centralizer of $f$} in the full diffeomorphism group, containing all of the powers of $f$.
 
 The question of whether an Anosov or partially hyperbolic $f$, when equipped with a large group of symmetries, is necessarily of algebraic origin has been studied extensively since the 1980's (see, e.g. \cite{KS94, KL, KalSpa, KSad, FKS,HW14, DX1, SpaVin, DSVX}) and is referred to as the {\it global rigidity} problem. Examples of systems of algebraic origin are automorphisms of tori and  translations on homogeneous spaces, all of which belong to the class of {\it affine diffeomorphisms}. 

From a somewhat different perspective, it has long been suspected that the generic diffeomorphism has no symmetries other than its own iterates, and Smale posed this formally as a question in \cite{Smale, Smale2}.  This question has been explored in various contexts since the 1970's 
\cite{Kopell, PY89} and
has been answered decisively in the $C^1$ topology on any manifold: the $C^1$-generic diffeomorphism has trivial centralizer \cite{BCWvolume, BCW}. 
 
 In this paper we unite these two perspectives in a program to describe  symmetry groups and their effect on dynamics in the setting of affine diffeomorphisms and their smooth perturbations.
This program is motivated by recent results about {\em centralizer rigidity} phenomena, first studied by three of the authors in \cite{DWXcent}.  

Particularly strong types of centralizer rigidity phenomena have been found to hold in the affine examples studied in \cite{ZhijingWang, Sven, Sandfeldt2024}.  For these examples, symmetries of {\em any} ($C^1$-small) smooth perturbation constitute a Lie group, there is a finite list of all possible Lie groups that may appear as symmetry groups of perturbations, and if the symmetry group remains (algebraically) the same after a perturbation then this forces rigidity: the perturbed system is smoothly conjugate to an algebraic one (i.e. to an affine diffeomorphism).  

What these affine examples in  \cite{Sven, Sandfeldt2024, ZhijingWang} (on certain nilmanifolds and semisimple homogeneous spaces, respectively) have in common is that they are {\em $K$-systems}, meaning that they preserve a finite Haar measure and with respect to this measure have no factors of entropy zero, a condition that implies mixing of all orders.
The classically-studied affine $K$-systems form a large subclass of affine diffeomorphisms that include the ergodic automorphisms of tori and diagonal translations on $\SL(n,\RR)$ quotients.
Affine $K$-systems  are partially hyperbolic, as are their perturbations,  and they exhibit a strong irreducibility condition known as {\em essential accessibility} in the partially hyperbolic literature.

The first step toward obtaining a general picture of the kind described above for arbitrary affine diffeomorphisms is taken in this paper: in our main result (Theorem~\ref{thm: liecent iff ergodic}), we show that the class of affine diffeomorphisms whose smooth symmetries remain affine after an affine perturbation is precisely the class of affine $K$-systems. Moreover, this class characterizes  those affine diffeomorphisms whose affine perturbations continue to have a (finite dimensional) Lie group as the group of all smooth symmetries. 

After presenting our main results,  we outline a program for the study of symmetries of perturbations of affine $K$-systems. Some of the conjectures we formulate here have been confirmed in earlier works \cite{DWXcent}, \cite{ZhijingWang}, \cite{Sven}, \cite{Sandfeldt2024} and some will  appear in the forthcoming works \cite{DSWX}, \cite{DWX rank 1 factor}.

In the following subsections, we present our main results, introducing the necessary background and notation.


\subsection{The centralizer}
The centralizer $\cZ^r(f)$ of a $C^r$ diffeomorphism $f\in \Diff^r(M)$ of a closed manifold $M$ (with $r\in \NN\cup \{\infty\}$) is simply the group of diffeomorphisms that commute with it:
\[
\cZ^r(f):= \{g\in \Diff^r(M): g\circ f = f\circ g\} < \Diff^r(M).
\]
As mentioned above, for any closed manifold $M$, it is conjectured that the $C^r$-generic\footnote{In the sense of Baire.} diffeomorphism $f\in \Diff^r(M)$ has the smallest possible centralizer -- meaning $\cZ^r(f) = \langle f \rangle := \{f^n: n\in \ZZ\}$, a property called {\em trivial centralizer}; this conjecture has been proved in certain settings, for example when $M=\RR/\ZZ$ \cite{Kopell}, for all $r\geq 2$, and for $M$ an arbitrary closed manifold and $r=1$ \cite{BCW}. 

For $f\in\Diff^\infty(M)$, we will call $\cZ^\infty(f)$  the {\em smooth centralizer} of $f$.
In the discussion that follows, we will primarily consider this case and will write $\Diff(M)$ for $\Diff^\infty(M)$ and $\cZ(f)$ for the smooth centralizer of $f$ when this context is understood.

\subsection{Affine diffeomorphisms}

Let $X=G/B$ be a compact homogeneous manifold, where $G$ is a connected Lie group acting transitively and locally freely on $X$, and $B$ is the point stabilizer of this action. 
We will assume that $X$ admits a $G$-invariant probability measure $\mu_{\Haar}$. Two main classes of examples are:
\begin{itemize}
    \item $B=\Gamma<G$ is a torsion-free, cocompact lattice, and\
    \item $G$ is compact and $B$ is any connected, closed subgroup. For example, $X= SO(n)/SO(n-1)\cong S^n$.
\end{itemize} 
 In this work, we will often use the notion of commensurability:  groups $H_1, H_2$ are {\em commensurable}, denoted here by $H_1\doteq H_2$, if they are isomorphic to finite index subgroups of some group $H$.  Commensurable groups are isomorphic ``up to finite groups."


An {\em affine diffeomorphism} $f_0\colon X\to X$ is a map of the form $L_a\circ \Phi$, where $L_a$ denotes left translation by
$a\in G$ and $\Phi\colon G\to G$ is an automorphism preserving $B$. 
We denote by $\Aff(X)$ the space of affine diffeomorphisms of $X$, endowed with the $C^0$ topology. Then  $\Aff(X)$ is a Lie group that embeds in the space $\Diff^r(X)$ of $C^r$ diffeomorphisms, and the $C^0$ and $C^r$ metrics coincide on $\Aff(X)$, for any $r\in \NN\cup \{\infty\}$.

Affine diffeomorphisms display a rich variety of dynamics, and their ergodic-theoretic properties with respect to the Haar measure $\mu_{\hbox{\tiny Haar}}$ have been studied extensively (see \cite{Dani77,BM81,KSS}).  Nonlinear perturbations and deformations of elements of $\Aff(X)$ within $\Diff^r(X)$ are also basic objects of study in smooth dynamics, from the structurally stable Anosov automorphisms of nilmanifolds to the non-structurally stable but {\em stably ergodic} (i.e., those whose $C^1$-small, volume-preserving perturbations in $\Diff^\infty(X)$ are all ergodic) partially  hyperbolic diffeomorphisms in, e.g.,  \cite{GPS, fede, BS97}.

The space $\Aff(X)$ is also  a natural class of diffeomorphisms to study with respect to their symmetries: all manner of Lie groups and many infinite dimensional groups appear as centralizers of affine diffeomorphisms.  While the centralizer of the generic element of $\Diff(X)$  is conjecturally trivial, this is decidedly not the case for affine diffeomorphisms:  for example, the smooth centralizer of the generic $f_0\in \Aff(\TT^3)$ is {\em nontrivial} (and in fact commensurable to the free abelian group $\ZZ^2$).

\subsection{Diffeomorphisms with Lie group centralizer.}
In this general homogeneous setting, we aim to understand  how the properties of $\cZ(f)$ vary for $f$ in a neighborhood $\Uu$ of a fixed $f_0\in \Aff(X)$, first for  $\Uu\subseteq \Aff(X)$ and then for  $\Uu\subseteq \Diff(X)$.   We start with the properties of $\cZ(f_0)$ itself.

Clearly the restricted centralizer $\cZ^{\Aff}(f_0):=  \{g\in \Aff(X): g\circ f = f\circ g\}$ within the group $\Aff(X)$ is a subgroup of  the smooth centralizer $\cZ(f_0)$.   On the other hand,  $\cZ(f_0)$ need not be a subgroup of $\Aff(X)$, or even be a Lie group (e.g., $f_0=\id$).  This leads us to the following problem.
\begin{problem}\label{p=affine centralizer}
Determine which $f_0\in\Aff(X)$ have the property that $\cZ(f_0) < \Aff(X)$.  More generally, for which $f_0\in \Aff(X)$ is $\cZ(f_0)$ a (finite-dimensional) Lie group?
\end{problem}
Previous works on this and related problems due to Walters \cite{Wa-example, Wa} and   Morris \cite{Wit87, Wit87cor}  show that the ergodic properties of $f_0$ are closely related to the structure of $\cZ(f_0)$.  
In this paper, we  give a  solution to  Problem~\ref{p=affine centralizer} by establishing a general connection between the ergodic-theoretic properties of $f_0$ and the properties of its centralizer.  


\medskip



\section{Main results: affine $K$-systems and stability of affine centralizers}
Our main results (Theorems~\ref{thm: liecent iff ergodic}, \ref{t=mixing ZLip is ZAff} and \ref{thm: not K imp big cent}) describe algebraic properties of the function $f\mapsto \cZ(f)$ on $\Aff(X)$.

To further motivate these results, consider the case in dimension 1. The group of affine diffeomorphisms of the circle $\TT=\RR/\ZZ$ is the semidirect product $\TT \ltimes \ZZ/2\ZZ$, whose action is generated by rotations and the involution  $x\mapsto -x$. Restricting to the orientation-preserving component $\Aff_+(\TT)$, the values of the function $f\mapsto \cZ(f)$ oscillate between $\TT$ (at the irrational rotations) and huge groups at the rational rotations.
In particular, if $f_0(x) = x+p/q$, then for any diffeomorphism $g\colon \TT\to \TT$, the lift $\tilde g$ of $g$ under the covering map $x\mapsto qx$ commutes with $f_0$ (since $\tilde g(x+ 1/q) = g(x) + 1/q$), and so the finite extension $\ZZ/q\ZZ\hookrightarrow \widetilde{\Diff(S^1)} \to \Diff(S^1)$ is contained in $\Z^\infty(f_0)$.


The unique ergodicity of the irrational rotations tightly controls the centralizer in this particular case.  Using Ratner's measure rigidity,  Morris \cite{Wit87} proved that a vast generalization of this phenomenon holds for the {\em unipotent} translations:  ergodicity of $f_0 = L_u$ with respect to $\mu_{\Haar}$ implies that $\cZ^1(f_0) < \Aff(X)$.\footnote{He in fact proves something stronger: for ergodic, unipotent $f_0$,  if $g$ preserves $\mu_{\Haar}$ and $gf_0=f_0g$, then $g$ is affine. For $g\in \cZ^1(f)$ (or even $\cZ^{\Lip}(f)$),  it is easy to see that the ergodicity of $f_0$ implies that $g$ preserves $\mu_{\Haar}$. }

In the general setting,  mere ergodicity of $f_0\in \Aff(X)$ does not imply that $\cZ(f_0)<\Aff(X)$ (see end of Section~\ref{ProofB} for an important example due to Walters).  Moreover, even when $\cZ(f_0)<\Aff(X)$ does hold, the condition can be fragile.
For example, while any ergodic rotation $f_0$ on $\TT$ does satisfy the condition $\cZ(f_0)<\Aff(\TT)$, the density of the (nonergodic) rational rotations in $\Aff_+(\TT)$ precludes the existence of a rotation $f_0$ where the condition $\cZ(f_0)<\Aff(X)$ holds stably under perturbations. 

Moving to another example,  when $f_0\in\Aff(\TT^d)$ is Anosov, the smooth centralizer of every perturbation of $f_0$ in $\Aff(\TT^d)$ is affine; these examples (unlike the irrational rotations) are also {\em stably} ergodic within  $\Aff(X)$.

It turns out these connections between the stability properties of ergodicity and affine centralizer is completely general:
Theorem~\ref{thm: liecent iff ergodic} below shows that
 stable ergodicity of $f_0$ within $\Aff(X)$ is in fact equivalent to  the stability of the condition  $\cZ <\Aff(X)$. To state this result, we define some terms.

An {\em affine $K$-system} is an affine diffeomorphism $f_0\in\Aff(X)$ with the Kolmogorov, or ``$K$", property with respect to $\mu_{\Haar}$.  The $K$ property means that $f_0$ has no measurable factor of entropy $0$, and it has many equivalent formulations, some of which we discuss below.  It is stronger than mixing and weaker than the Bernoulli property.  Conjecturally, all affine $K$-systems are expected to be exponentially mixing. 

Let $\Aa(X) = \{ f\in \Aff(X): \cZ(f)< \Aff(X)\}$ be the set of affine diffeomorphisms whose smooth centralizer is affine, and let \[
\Ll(X) :=  \{ f\in \Aff(X): \hbox{$\cZ(f)$ is a Lie subgroup of $\Diff(X)$} \}.
\]
Our main result classifies $\Ll(X)$ and the $C^0$-interior $\interior(\Aa(X))$ in terms of the ergodic properties of their elements.

To state our main result, we also introduce the following notation for the sets $\Ee(X), \Ww(X), \Kk(X) \subset \Aff(X)$
of affine ergodic, weakly mixing and $K$-systems, respectively.

\begin{itemize}
    \item $\Ee(X) := \{f \in \Aff(X): \hbox{$f$ is ergodic (with respect to $\mu_{Haar}$)}\}$.
     \item $\Ww(X) := \{f \in \Aff(X): \hbox{$f$ is weakly mixing (with respect to $\mu_{Haar}$)}\}$.
    \item $\Kk(X) := \{f \in \Aff(X): \hbox{$f$ is a $K$-system (with respect to $\mu_{Haar}$)} \}$.
\end{itemize}
Note that $\Kk(X)\subset \Ww(X) \subset \Ee(X)$.

 Brezin-Shub \cite{BS97} and later Starkov \cite[Appendix]{PS04}  investigated the elements of $\Aff(X)$ that are stably ergodic within $\Aff(X)$ -- that is,  the $C^0$ interior $\interior(\Ee(X))$ in $\Aff(X)$.  Their work, combined with earlier work of Dani \cite{Dani77}, shows that $\Kk(X)$ is open in $\Aff(X)$, and in fact
 \[\Kk(X) =  \interior(\Ww(X)) = \interior(\Ee(X)).\] Our main result identifies the interior of $\Aa(X)$ as this common set.

\begin{maintheorem}\label{thm: liecent iff ergodic} For any compact homogeneous manifold $X$, we have
\[\Ee(X) \subseteq \Ll(X),\]
and 
\[\Kk(X) = \interior(\Aa(X)) =  \interior(\Ll(X)) .\]
 
%


\end{maintheorem}

Our proof of Theorem~\ref{thm: liecent iff ergodic} uses the following characterization of $\Kk(X)$ from \cite{Dani77} (see also \cite{BS97}). For an affine diffeomorphism $f_0$ of $X=G/B$, consider the induced automorphism $\hat f_0 := ad(g)\circ D_e\Phi\colon \mathfrak{g} \to \mathfrak{g}$ of the Lie algebra $\mathfrak{g}$ of $G$.  Let $V^{su}\subset \mathfrak{g}$ be the direct sum of the generalized eigenspaces corresponding to those eigenvalues $\lambda$  of $\hat f_0$ with $|\lambda|\neq 1$, and let $\mathfrak{h}\subset \mathfrak{g}$ be the ($\hat f_0$-invariant) Lie subalgebra generated by $V^{su}$.  
Then $\mathfrak{h}$ is an ideal and thus tangent to a connected normal subgroup $H\lhd G$, which we call the {\em hyperbolically generated subgroup of $G$}. 

In their main result, \cite{BS97} show that if $f_0\in \interior(\Ee(X))$ then $\overline{HB} = G$. 
In fact, the property $\overline{HB} = G$ is stable under affine perturbations, and so $f_0\in \interior(\Ee(X))$ is {\em equivalent} to   $\overline{HB} = G$ \cite[Appendix, Theorem 1]{PS04}. On the other hand, Dani had previously proved that the condition $\overline{HB} = G$ is equivalent to $f_0$ being a $K$-automorphism.
In the language of smooth dynamics, the affine $K$-systems are prescisely  those  that are partially hyperbolic and essentially accessible.

Since any $f_0\in \interior(\Ee(X))$ is a $K$-system, and hence weakly mixing, the proof of the inclusion $\interior(\Ee(X)) \subseteq \interior(\Aa(X))$ in Theorem~\ref{thm: liecent iff ergodic} is implied by the following.  
For $f_0\in \Aff(X)$, we denote by $\Z^{\Lip}(f_0)$ the set of bi-Lipschitz homeomorphisms of $X$ commuting with $f_0$.  Clearly  $\Z^\infty(f_0) < \Z^{\Lip}(f_0)$.

\begin{maintheorem}
\label{t=mixing ZLip is ZAff}
If $f_0\in \Aff(X)$ is ergodic, then $\Z^{\Lip}(f_0)= \Z^\infty(f_0)$  is a Lie subgroup of $\Diff^{\infty}(M)$. In particular,
\[
\Ee(X) \subseteq \Ll(X).
\]

If $f_0\in \Aff(X)$ is weakly mixing, then $\Z^{\Lip}(f_0) = \Z^{\Aff}(f_0)$. In particular,
    \[
   \Ww(X) \subseteq \Aa(X).
    \]

\end{maintheorem}

\begin{remark}  Morris \cite{Wit87, Wit87cor} proved a stronger result in the weak mixing,  entropy $0$ context: for $f_0\in \Aff(X)$ with entropy 0 with respect to $\mu_{\Haar}$, if $f_0$ is weakly mixing and commutes with a $\mu_{\Haar}$-preserving measurable map $g$, then $g$ is affine.\footnote{This implies Theorem~\ref{t=mixing ZLip is ZAff} in the case where $f_0$ is weakly mixing and has entropy $0$, since if $f_0$ is ergodic with respect to $\mu_{\Haar}$, and $g$ is bi-Lipschitz and commutes with $f_0$, then $g$ must preserve $\mu_{\Haar}$ as well.} 

Walters also proved a stronger result in the case where $X=N/B$, with $N$ nilpotent: if $f_0\in\Aff(X)$ is weakly mixing, then $\cZ^0(f_0)<\Aff(X)$, where $\cZ^0(f_0)$ denotes the group of homeomorphisms commuting with $f_0$.  Note that this case is complementary to the situation considered in \cite{Wit87, Wit87cor}, as weakly mixing maps on nilmanifolds must have positive entropy. 

For an ergodic homogeneous semisimple flow, \cite[Theorem C.1]{DSVX} also derive that the Lipshitz centralizer of the flow is affine, from earlier work of Zeghib \cite{Zeghib}.
\end{remark}
\begin{remark} The bi-Lipschitz centralizer $\Z^{\Lip}(f_0)$ can be replaced by the set of homeomorphisms $g$ commuting with $f_0$
with the properties that $g,g^{-1}$ be differentiable a.e.~with $\|Dg\|, \|Dg^{-1}\| \in L^2(X,\mu_{\Haar})$: the proof is almost identical.  One can also extend this result to conjugacies between affine diffeomorphisms: if $f_1,f_2\in \Aff(X)$, with $f_1$ weakly mixing,  and $f_2 = g\circ f_1 \circ g^{-1}$, for some bi-Lipschitz homeomorphism $g$, then $g\in \Aff(X)$. 
\end{remark}
\begin{remark}
    If $f_0$ is ergodic but not weakly mixing, then it is possible to have $\cZ^\infty(f_0)\not\subset \Aff(X)$ : we describe  examples, originally found by Walters, 
    after the proof of  Theorem~\ref{t=mixing ZLip is ZAff} at the end of Section \ref{ProofB}.
    For these examples, $\Z^{\Aff}(f_0)$ is a proper subset of  $\Z^\infty(f_0)$, which is nonetheless a  Lie subgroup of $\Diff^\infty(X)$.  The existence of such examples is made clear by the proof of Theorem~\ref{t=mixing ZLip is ZAff}.

    
    
\end{remark}
\begin{remark} 
    In certain circumstances,  ergodicity of $f_0$ alone implies $\Z^{\Lip}(f_0)<\Aff(X)$; see Remark~\ref{remark 10} in the proof of Theorem~\ref{t=mixing ZLip is ZAff}.
\end{remark}

To prove the inclusion $\interior(\Aa(X)) \subseteq \interior(\Ee(X))$ in Theorem~\ref{thm: liecent iff ergodic}, we establish the contrapositive:
 
\begin{maintheorem}\label{thm: not K imp big cent} If $f_0\in \Aff(X)$ is not a K-system, then there exists an affine perturbation $f$ of $f_0$ such that $\Z^\infty(f)$ is not a Lie group and in fact  contains 
a continuously embedded copy of one of the following groups:
\begin{itemize}
\item the abelian group $C^\infty_c(I)$ of compactly-supported smooth functions on the interval $I=(0,1)$ (under addition) or  
\item the simple group $\Diff^\infty_c(I)$ of compactly-supported diffeomorphisms of $I$ (under composition).
\end{itemize}
\end{maintheorem}

 Hence $\interior(\Ll(X))\subseteq \interior(\Ee(X))$
 Since $\Aa(X) \subseteq \Ll(X)$, Theorems~\ref{t=mixing ZLip is ZAff} and \ref{thm: not K imp big cent} imply Theorem~ %


\subsection{Structure of the proofs}\label{ss=structure}

Here we sketch an outline of the proofs.
\medskip

\noindent
{\bf Theorem ~\ref{t=mixing ZLip is ZAff}.}
The proof of  Theorem~\ref{t=mixing ZLip is ZAff} is relatively straightforward. We suppose that $f_0\in\Aff(X)$ is weakly mixing with respect to $\mu_{\Haar}$ and fix an appropriate framing of the complexified tangent bundle $T_\CC X$ with the following property: the derivative cocycle (with respect to this framing) of a diffeomorphism $g$ of $X$ is constant if and only if $g$ is affine (with respect to the fixed homogeneous structure on $X$).
This framing is further chosen so that the (constant) matrix cocycle  $Df_0$ is in Jordan canonical form.

Fixing $g\in \cZ^{\Lip}(f_0)$, we explore the consequences of the ($\mu_{\Haar}$-almost everywhere) cocycle relations $D(g\circ f_0^n) = D(f_0^n\circ g)$, for $n\in\ZZ$, obtaining asymptotic twisted invariance equations over $f_0$ that the entries of $Dg$ must satisfy a.e.  The assumption that $f_0$ is weakly mixing then implies that the entries of $Dg$ are constant a.e., which implies that $g\in \Aff(X)$.

In the case where $f_0$ is merely ergodic, we further explore the consequences of these asymptotic twisted invariance equations over $f_0$, using the simplicity of eigenspaces of the associated Koopman operator and smoothness of the associated eigenfunctions to embed $\cZ^{\Lip}(f_0)$ into $\GL(d,\CC)$.
\bigskip

\noindent
{\bf Theorem~\ref{thm: not K imp big cent}.}
To prove Theorem~\ref{thm: not K imp big cent}, we consider $f_0\in \Aff(X)$ that is not a $K$-system. Let $H \triangleleft G$ be the hyperbolically-generated subgroup for $f_0$; by assumption, we have $\overline{H\cdot B}\neq G$.
This gives a nontrivial fibration
\[(\overline{H\cdot B})/B\to G/B \to G/(\overline{H\cdot B})
\cong\dfrac{G/H}{(\overline{H\cdot B})/H}.\]
Crucially, the projection of $f_0$ to $G/(\overline{H\cdot B})$ is an affine diffeomorphism all of whose eigenvalues have modulus $1$. 

Using the structure theory of connected Lie groups, we then construct $g_0\in G$ arbitrarily close to the identity element $e$ such that the affine perturbation
$f = L_{g_0} f_0$  carries a nontrivial, $f$-invariant smooth fiber bundle structure
$F\to G/B \stackrel{\pi}{\to} \bar X$ over a compact homogeneous space $\bar X$, with  fiber $F = G_F/B_F$ a compact homogeneous space, satisfying the following properties:
\begin{enumerate}
\item the projection $\bar f\colon \bar X\to \bar X$ of $f$ is a periodic affine map, of period $d\geq 1$; and
\item there exist a nonempty open set $U\subset \bar X$  whose iterates $f(U), \ldots, f^{d-1}(U)$ are disjoint, and a trivialization \[\pi^{-1}(\bigsqcup \bar{f}^j(U)) \cong \bigsqcup \bar{f}^j(U) \times F\] such that, in these trivializing coordinates, the restriction ${f^d_{\bar x}}$ of $f^d$ to the fiber $F_{\bar x} = \pi^{-1}(\bar x)$ 
is affine, with constant linear part: for every $\bar x\in\bar U$ and $y\in F_{\bar x}$,
\[f^d_{\bar x}(y):= b(\bar x) A(y),\] for some smooth $b \colon U \to G_F$ and automorphism $A$ of $G_F$ preserving $B_F$ (which crucially does not depend on $y$). 
\end{enumerate}

Having carried out this construction, the proof breaks into two cases:
\begin{itemize}
\item[]{\bf Case 1: 
$1$ is an eigenvalue of $D_eA$.}  In this case, we construct a $Df$-invariant vector field $V$ tangent to the fibers of $\pi$ that
generates a flow $\psi_t$  commuting  with $f$. Moreover, there is a nonempty, $\bar f$-invariant open set ${\mathcal V}\subset \bar X$ so that $V$ is nonvanishing on $\pi^{-1}({\mathcal V})$.
(Proposition~\ref{prop: key prop 1}).
\smallskip
\item[]{\bf Case 2: 
$1$ is not an eigenvalue of $D_eA$.}  We show that 
the dynamics of $f$ restricted to $\pi^{-1}(\bigsqcup \bar{f}^j(U)) \cong \bigsqcup \bar{f}^j(U) \times F$ is then a product, i.e. 
in these trivializing coordinates, $f$ is smoothly conjugate to a product map
\[(\bar x, y) \mapsto (\bar f(\bar x), h(y)),\] for some $h\in \Aff(F)$ (Proposition~\ref{prop: key prop 2}).
\end{itemize}

In both cases, we claim that the smooth centralizer of  $f$ is not a Lie group.

To see this in Case 1, note that any compactly supported, smooth function $\phi\colon \bar X\to \RR$  that is periodic under the action of $\bar f$
pulls back to a compactly supported, $f$-invariant smooth function $\tilde\phi\colon X\to \RR$. Then the 
vector field $\tilde\phi V$ is also $Df$-invariant (since on each fiber $F_{\bar x}$, $\tilde \phi  V$ is a scalar multiple of $V$). The flow generated by $\tilde\phi V$ is thus contained in $\cZ(f)$, for every such $\psi$, which implies that $\cZ(f)$ is not a Lie group.  Indeed,  one can smoothly embed the interval $I$ in ${\mathcal V}$ so that it is disjoint from its $\bar f$ iterates.  In this way, $C^\infty_c(I)$ can be embedded in $\cZ^\infty(f)$.

In  Case 2, we fix any compactly-supported diffeomorphism $\bar g\colon U\to U$, which extends to a compactly-supported diffeomorphism $\bar g\colon \bigsqcup \bar{f}^j(U)  \to \bigsqcup \bar{f}^j(U)$ commuting with $\bar f$. With respect to the trivializing coordinates above, we then define $g\colon \pi^{-1}(\bigsqcup \bar{f}^j(U))  \to \pi^{-1}(\bigsqcup \bar{f}^j(U))  $ by
\[ g = \bar g\times \Id_F\]
and extend the diffeomorphism by the identity on the rest of $X$ (see Lemma~\ref{wildperiod}). This again gives an uncountably-generated family of diffeomorphisms in $\cZ(f)$. Embedding $I^k$ in $U$  (for $k=\dim U$) thus gives an embedding of  $\Diff^\infty_c(I^k)$ in $\cZ(f)$, from which one obtains an embedding of $\Diff^\infty_c(I)$ into $\cZ(f)$.

\section{A perturbative study of symmetries: the central program}

 In the setting where $M=X$, we are interested in the {\em algebraic} properties of the values of the function $f\mapsto \cZ(f)$ in the neighborhood of a  fixed $K$-automorphism  $f_0\in \Aff(X)$. In the previous section we considered $\cZ$ as a function on the domain $\Aff(X)$; subsequently we consider it as a function on the full diffeomorphism group $\Diff(X)$. The focus of the program is to examine the interplay between the algebraic properties of $\cZ(f)$ and the dynamical properties of $f$.



Our starting point and motivational example is a global rigidity theorem of Rodriguez Hertz and Wang
  \cite{HW14}, which  directly implies the following perturbative result.
\begin{theorem} 
\label{t=FRHW}\cite{HW14}
Let $d\geq 3$ and $f_0\in \Aff(\TT^d)$ be of the form  $f_0(x) = Ax + b$ with $A\in\SL(d,\ZZ)$ hyperbolic and irreducible (i.e., the characteristic polynomial of $A$ does not split over $\QQ$).  Then for any $f\in \Diff^\infty(\TT^d)$ sufficiently $C^1$-close to $f_0$,   either $\cZ(f)$ is virtually trivial, or $f$ is $C^\infty$ conjugate to $f_0$. 
\end{theorem}
Theorem~\ref{t=FRHW} gives a complete description of the behavior of $f\mapsto \cZ(f)$ in a neighborhood of an irreducible hyperbolic automorphism $f_0\in\Aff(X)$: up to commensurability, $\cZ(f)$ assumes only two values, $\ZZ$ and $\cZ(f_0)$ (which is commensurable to $\ZZ^\ell$, for some $\ell\geq 2$), and moreover when this maximal value $\cZ(f)\cong \cZ(f_0)$ is attained, there is strong rigidity: $f$ is smoothly affine.

The program we outline here explores to  what extent Theorem~\ref{t=FRHW} generalizes to affine diffeomorphisms of a general compact homogeneous manifold $X$.

\subsection{From affine perturbations to general perturbations.}

While it is natural to ask whether an analogue of Theorem B holds under smooth perturbations—i.e., whether ergodicity (or topological transitivity) still guarantees a Lie group centralizer—this turns out to be false. Even in a neighborhood of affine diffeomorphisms, counterexamples arise: for instance, the centralizer of a weakly mixing pseudo-rotation on the sphere fails to be locally compact due to local rigidity phenomena \cite{AFLXZ}. While these examples are perturbations of systems (rotations) with Lie group centralizer, the {\em dynamical reason} that rotations have Lie group centralizer is not robust under perturbations.  

On the other hand, the affine $K$-systems discussed in this paper, while not structurally stable, have many {\em robust} dynamical features, first and foremost being partial hyperbolicity. Moreover, {\em accessibility}\footnote{A partially hyperbolic diffeomorphism $f$ is {\em accessible} if there exists an $su$-path -- that is, a piecewise $C^1$ path everywhere tangent to either the stable or unstable bundle -- connecting any two points in $M$.  $f$ is {\em essentially accessible} if any two positive measure sets can be connected by such a path. All affine $K$-systems are essentially accessible.}, a property that many affine $K$-systems have (for example those for which $G$ is simple), is also stable under perturbation, and this property is closely connected to that of having Lie group centralizer, as we discuss below. 



Our main result, Theorem \ref{thm: liecent iff ergodic}, identifies these affine $K$-systems as precisely those affine diffeomorphisms 
whose centralizer remains a Lie group after affine perturbations. We conjecture  that the $K$-systems continue to have Lie group centralizer when perturbed within $\Diff^\infty(X)$. 




\begin{conjecture}\label{conj: lie group rank 1} Let $f_0$ be an affine diffeomorphism of a compact homogeneous manifold $X=G/B$ admitting a $G$-invariant probability measure. 

Then $f_0$ is a $K$-system (if and) only if for every $f\in \Diff^\infty(X)$  sufficiently $C^1$-close to $f_0$,
the centralizer $\Z^\infty(f)$ is a $C^0$-closed Lie subgroup of $\Diff^\infty(X)$.
\end{conjecture}
One direction of the conjecture, the ``if and" part, follows directly from Theorem~\ref{thm: liecent iff ergodic}.
%
A slight generalization of  the results in $\S$\ref{ss=support}) establishes the  ``only if'' part of this conjecture for the accessible $K$-systems.  This leaves open the general case of $K$-systems that are essentially accessible but not accessible.  While for general partially hyperbolic systems, essential accessibility is not stable (see \cite{Dolgopyat}), it remains open whether this property might be stable for affine systems.  Indeed, in some situations this stability has been established (see \cite{fede}).  In other contexts, something weaker than essential accessibility is enough to obtain Lie group centralizer (see \cite{DWXcent}). 

A natural follow-up question is whether perturbations of affine 
$K$-systems admit a {\em uniform} upper bound on the dimension of their centralizers. More generally:  what dynamical properties of a diffeomorphism $f$ guarantee an {\em a priori} upper bound on the dimension of any Lie group contained in $\cZ(f)$? Quite generally, topological transitivity gives such a bound.

\begin{prop}\label{T:bounddim}
Let $M$ be a topological manifold and $h:M\to M$ 
with a point $x\in M$ whose orbit under $h$ is dense in $M$. 
Then the $C^0$ centralizer $\cZ^0(h)$ cannot contain a Lie group $G$ with $\dim(G)>\dim(M)$.
\end{prop}
\begin{proof}
    Suppose to the contrary that there exists a Lie subgroup $G$ of $\cZ^0(h)$ with $\dim(G)>\dim(M)$. It is clear that $G$ acts on $M$ via homeomorphisms. Let $x\in M$ be a point whose orbit closure under $h$ equals $M$, and $G_x$ be the stabilizer of $x$ in $G$. If $G_x$ is trivial, then the orbit map $G\to M$ via $g\mapsto gx$ is an injective continuous map. It follows from invariance of domain that $\dim(G)\leq\dim(M)$, a contradiction. So $G_x$ must be non-trivial. Let us take any non-identity element $\phi\in G$ with $\phi(x)=x$. It follows from $\phi\circ h=h\circ\phi$ that $\phi$ also fixes $h^n(x)$ for any $n\in\ZZ$. By the density of $\{h^n(x):n\in\ZZ\}$ in $M$ and the continuity of $\phi$, we conclude that $\phi=\id$, which is a contradiction.
\end{proof}

Proposition~\ref{T:bounddim} cannot however be applied in general to  perturbations of affine $K$-systems, as there exist (accessible!) examples  that are not topologically transitive (see \cite{SThs}).  For  accessible affine $K$-systems, the same proof that gives that the centralizer of any perturbation is a Lie group also gives a uniform bound on its dimension: it is at most the dimension of the center bundle in the partially hyperbolic splitting.

It turns out that for the general class of partially hyperbolic diffeomorphisms,  a property weaker than essential accessibility (which might hold for perturbations of affine $K$-systems) suffices to obtain such a bound.
\begin{prop}\label{T:bounddim PH}
Let $M$ be a smooth manifold and $f\in \Diff(M)$ be a partially hyperbolic diffeomorphism 
with a point $x\in M$ whose accessibility class \[\mathrm{acc}(x):= \{y\in M: \hbox{$\exists$ an $su$-path from $x$ to $y$} \}\] is dense in $M$. 
Then the $C^1$ centralizer $\cZ^1(f)$ cannot contain a Lie group $G$ with $\dim(G)>\dim(E^c_f)$.
\end{prop}
\begin{proof}We start with a lemma.
\begin{lemma}\label{lem: bd PH} Let $V$ be a continuous $Df$-invariant vector field.  Then
\begin{enumerate}
    \item for every $p\in M$, $V(p)$ is tangent to $E^c(p)$; and
    \item if  $x\in M$ has a dense accessibility class, then  $V(x)=0$ implies that $V=0$ everywhere.
\end{enumerate}\end{lemma}
\begin{proof}(1). Assume $V(p)$ is not tangent to $E^c(p)$. The invariance $Df^n_\ast V=V$ implies that $V(f^n(p))=Df^n(V(p))$. By taking $n\to \pm \infty$, we obtain that $V$ is not bounded, contradicting  the continuity of $V$.

(2). We claim that if there exists a $p$ such that $V(p)=0$, then $V(q)=0$ for any $q$ in $W^s(p)$ (or $W^u(p)$). Suppose $V(q)\neq 0$, and consider the flow $\varphi_t$ generated by $V$; then for $t$ small, $\varphi_t(q)\neq q$. By commutativity, $\varphi_t$ preserves every  $W^s$-leaf. Since $V(p)=0$, $\varphi_t(p)=p$. Hence $$W^s(\varphi_t(q))=\varphi_t\cdot W^s(q)=\varphi_t\cdot W^s(p)=W^s(\varphi_t(p))=W^s(p)=W^s(q).$$
Therefore $\varphi_t(q)$ belongs to  $W^s(q)$ for small $t$, and so $V(q)$ is non-zero and tangent to $E^s(q)$. By (1) we know this is impossible. The claim holds. 

Now for a point $x\in M$ with a dense accessibility class, if $V(x)=0$ then $V$ is $0$ on $\mathrm{acc}(x)$, and so by density of $\mathrm{acc}(x)$, we have $V=0$ everywhere.
\end{proof}
Suppose now that there exists a Lie subgroup $G$ of $\cZ^1(f)$ with $\dim(G)>\dim E^c_f$. Let $x$ be a point with a dense accessibility class, and let $G_x$ be the stablizer of $x$ in $G$, which is closed, and hence is also a Lie group. If $G_x$ is discrete, then the orbit map $G\to M$ via $g\mapsto gx$ is a local $C^1$ embedding, whose tangent map embeds $\mathrm{Lie}(G)$ into $E^c(x)$. Thus $\dim G\leq \dim E^c_f$, contradiction. If $G_x$ is non-discrete, then there is a non-vanishing continuous vector field $V$ (an element in $\mathrm{Lie}(G_x)$) that is $Df$-invariant and vanishing at $x$. This contradicts  Lemma \ref{lem: bd PH}.
\end{proof}
 
\subsection{Rigidity and rank-1 factors.}   
In the scenario posited by Conjecture~\ref{conj: lie group rank 1}, the first natural question to ask is whether an analogue of the strong rigidity result of \cite{HW14} in Theorem~\ref{t=FRHW} holds in this context: is it the case that for a perturbation $f$ of an  affine $K$-system, either $\cZ(f) \doteq \langle f \rangle$ or $f$ is smoothly affine (i.e.~smoothly conjugate to an affine diffeomorphism)?  

 The answer to this question is ``no" without some type of irreducibility assumption on $f_0$ like the one in Theorem~\ref{t=FRHW}. In particular, if $\cZ(f_0)$ has a {\em smooth rank-$1$ factor}, meaning that there is a smooth submersion $X\to Y$ that semiconjugates the action of $\cZ(f_0)$ to an action on $Y$ of a group commensurable to $\RR$ or $\ZZ$, then it is often possible to perturb $f_0$ to a non-affine $f$ without changing the centralizer, even when the centralizer is reasonably ``large." 
 
 For example, if 
 $f_0 = g_0\times h_0 \in \Aff(\TT^{d+2})$ is the product of hyperbolic automorphisms $g_0\in\Aff(\TT^2)$ and $h_0\in\Aff(\TT^d)$ with $d\geq 2$, then any perturbation of the form $f=g\times h_0$ has  $\cZ(f)\doteq \cZ(f_0)\doteq \ZZ^\ell$, for some $\ell\geq 2$, but most such perturbations are not smoothly affine. Note that projection onto the first $\TT^2$-factor semiconjugates the action of $\cZ(f_0)$ to the rank-1 action of $\cZ(g_0)\doteq \ZZ$, and this rank-1 factor is the source of the non-rigidity of this example.

 \begin{definition}[Smooth rank-$1$ factor]\label{def: smooth rank-1 factor} Let $H$ be a topological group and let $\alpha\colon H\to \Diff^\infty(M)$ be a continuous action on a closed manifold $M$. A ($C^\infty$) {\em rank-1 factor of $\alpha$} is defined by 
\begin{itemize}
\item a closed $C^\infty$ manifold $Y$ and a $C^\infty$
submersion $\pi\colon M\to Y$;
\item a surjective (topological group) homomorphism $\sigma\colon H\to A$, where $A$ is a compact extension of $\RR$ or $\ZZ$; 
and
\item a locally free, faithful 
$C^\infty$ action $\bal\colon A\to \Diff^\infty(Y)$ such that \[\pi(\alpha(a))(x) = \bal(\sigma(a))(\pi(x)),\]
for all $x\in M, a\in H$. 
\end{itemize}
\end{definition}

We discuss in Remark~\ref{rem: totally transitive} algebraic conditions on $f_0\in\Aff(X)$ that imply that $\cZ(f_0)$ has no smooth (or even continuous) rank-1 factors.  Examples (which are also $K$-systems) are ergodic automorphisms of tori induced by irreducible matrices and nontrivial diagonal translations on $\SL(n,\RR)$ quotients, for $n\geq 3$.

We conjecture that a centralizer with a rank-1 factor is the only obstruction to the following type of smooth centralizer rigidity.

\begin{conjecture}\label{conj: main smooth variant} Let $f_0$ be an affine $K$-system on a compact homogeneous manifold $X$. Assume that $\Z^\infty(f_0)$ has no rank-1 factor \footnote{may be substituted by a weaker condition: there exists a free abelian subgroup of $\Z^\infty(f_0)$ which has  rank 2, contains $f$ and has no rank one factors.}. If $f$ is sufficiently $C^1$-close to $f_0$ in $\Diff^\infty(X)$, then either $\Z^\infty(f)$ has a smooth rank-1 factor, or $f$ is smoothly conjugate to an affine $K$-system.
\end{conjecture}

 We note that the irreducibility assumption on $f$ in Theorem~\ref{t=FRHW} is stronger than the ``no rank-1 factors" assumption, but the main result in \cite{HW14} also implies the conclusion of the theorem under this weaker hypothesis.



 The following variation on Conjecture~\ref{conj: main smooth variant} is a more algebraic take on Theorem~\ref{t=FRHW}.  We again first consider the special case where $X=\TT^d$ and $f_0$ is hyperbolic.  Structural stability of $f_0$ then implies that every smooth perturbation $f$ is topologically conjugate to $f_0$ by some homeomorphism $h$, and so if $\phi\in \cZ(f)$,  then $h\phi h^{-1}$ is a homeomorphism commuting with $f_0$.  But any such homeomorphism must be affine \cite{AP} and so $\cZ(f)\cong h \cZ(f) h^{-1}$ is a subgroup of $\cZ(f_0)$.  The maximal such subgroup is of course $\cZ(f_0)$ itself, and it is natural to explore what happens when $\cZ(f) \cong \cZ(f_0)$.

In this hyperbolic setting, the results in \cite{HW14} also  imply the  statement: if $\cZ(f_0)$ has no rank-1 factors, and  $\cZ(f)\cong \cZ(f_0)$, then the perturbation $f$ is smoothly affine.  In the general case, when an affine $K$-system $f_0$ is not structurally stable,   $\cZ(f)$ need not be a subgroup of $\cZ(f_0)$, even when the perturbation $f$ is also affine.
Nonetheless, it is reasonable to expect that $\cZ(f) \cong \cZ(f_0)$ (or the weaker statement  $\cZ(f) \doteq \cZ(f_0)$) to be a rigid scenario, motivating our next conjecture.
 
 \begin{conjecture}\label{conj: main smooth} Let $f_0$ be an affine $K$-system on a compact homogeneous manifold $X$.  Assume that $\Z^\infty(f_0)$ contains an embedded $\mathbb Z^2$ subgroup that has no rank-1 factors.
 
 If a diffeomorphism $f$ is sufficiently $C^1$ close to $f_0$ in $\Diff^\infty(X)$, 
and $\Z^\infty(f) \doteq \Z^\infty(f_0)$ (as abstract groups), 
then $f$ is $C^\infty$-conjugate to an affine $K$-system on $X$.
\end{conjecture}


\begin{remark}\label{rem: totally transitive} 
We note here that while the condition of ``not having a rank-1 factor" that appears in these conjectures (and more generally in the study of abelian actions) might appear difficult to check, there are natural conditions that imply it. 

An action by a (discrete) group of transformations in which all transformations  (except for the identity)  have some property $\mathcal P$ is said to be  a {\it totally} $\mathcal P$ {\it action}. If the centralizer $\Z^\infty(f)$ of $f$ contains a totally topologically transitive $\mathbb Z^2$-subgroup then  $\Z^\infty(f)$ has no (topological) rank 1-factors. 
Namely, if there is a rank-1 factor and there exists an epimorphism $\sigma$ as in Definition \ref{def: smooth rank-1 factor}, then there is a transitive element $g$ of the action in the kernel of $\sigma$. But then if $p$ has a dense $g$-orbit, then $\sigma^{-1}(\sigma(p))$ is a closed, $g$-invariant set with a dense orbit in $X$: this implies that the factor map is trivial, sending $X$ to $\sigma(p)$.
For this argument it suffices that $\pi$ is continuous, and thus total topological transitivity implies that there are no {\it continuous} rank-one factors. 

Since ergodicity with respect to $\mu_{\Haar}$ implies topological transitivity, if $f_0$ is $K$, and $\Z^{\rm Aff}(f_0)$ contains a totally ergodic action by a $\mathbb Z^2$ subgroup generated by $f_0$ and some $g_0$ (let us denote this subgroup by $\langle f_0, g_0\rangle$), then $\Z^{\rm Aff}(f_0)$ has no rank-1 factors. 
\end{remark}

\begin{remark}

Let $f_0\in \Aff(X)$ be an affine $K$-system. If $f$ is a $C^1$-small perturbation of $f_0$ and $\Z^\infty(f)$ intersects non-trivially a $C^1$-small neighborhood of $g_0$ at some smooth diffeomorphism $g$, then the action $\langle f, g\rangle$ is a $C^1$-small perturbation of $\langle f_0, g_0\rangle$. The question of whether any  small perturbation of action  $\langle f_0, g_0\rangle$  is smoothly conjugate to an affine action is known in literature as a {\it local rigidity problem} and there are many results in this direction. We stress however that despite large literature in this direction, for general  totally K actions $\langle f_0, g_0\rangle$ that are not Anosov, the answer to the local rigidity problem is still not completely known. The known cases are either accessible (\cite{DK-PH}, \cite{VW}, \cite{Sandfeldt2024}), or require control on perturbations in high regularity (\cite{DK-KAM}, \cite{ZhenqiWang}). The latter have been obtained via a KAM method, which cannot possibly give results for $C^1$-small perturbations. An answer to Conjecture \ref{conj: main smooth variant} would in fact imply local rigidity for $C^1$-small totally ergodic perturbations of totally K affine $\ZZ^2$-actions. 
\end{remark}

\subsection{Topological rigidity.}
Finally, we further explore the question: ``what happens when $\cZ(f)\doteq \cZ(f_0)$?" from a different perspective. Conjectures~\ref{conj: main smooth variant} and \ref{conj: main smooth}  address the case when $f_0$ has no rank-$1$ factors, but what if $f_0$ does have a rank-$1$ factor?  In the simple example of a hyperbolic $f_0\in\Aff(\TT^2)$, and any perturbation $f$,  $\cZ(f)$ is a subgroup of $\cZ^0(f_0)= \cZ^{\Aff}(f_0)\doteq \ZZ$, and so the condition $\cZ(f) \doteq \cZ(f_0)$ clearly does {\em not} imply that $f$ is smoothly conjugate to an affine diffeomorphism. Yet in this example, $f$ is {\em topologically} conjugate to an affine diffeomorphism, namely the automorphism $f_0$ itself.

Our final conjecture below addresses a generalization of this situation, in which our given affine $K$-system $f_0$ commutes with an affine hyperbolic map $g_0$.  We conjecture that under this scenario, any perturbation $f$ of $f_0$ with $\cZ(f)\doteq \cZ(f_0)$ must be topologically conjugate to an affine system (even though $f_0$  itself is not necessarily hyperbolic and so {\em not} a priori structurally stable,  and the centralizer of $f$ a priori need not contain any Anosov elements).

We consider affine $K$-systems
whose centralizer contains an Anosov diffeomorphism. This condition neither implies nor is implied by the ``no rank-1 factors" condition.  
\begin{conjecture}\label{conj: main top}Let $f_0\in \Aff(X)$ be a  $K$-system. Assume that $\cZ(f_0)$ contains a hyperbolic affine diffeomorphism. Suppose that  $f\in \Diff(X)$ is sufficiently $C^1$-close to $f_0$ and $\cZ(f)\doteq \cZ(f_0)$. Then $f$ is $C^0$ conjugate to an affine $K$-system.
\end{conjecture}

\subsection{Support for and approach to the conjectures}\label{ss=support}
Conjectures \ref{conj: lie group rank 1}-\ref{conj: main top} are supported by a variety of results, some old and some very recent. 
 
The assumptions on $f_0$ and $Z(f_0)$ in our conjectures place restrictions on the homogeneous spaces $X$ that one needs to consider. The strongest restriction appears in Conjecture~\ref{conj: main top}, where $X = G/B$ must carry an Anosov automorphism, which forces $G$ to be nilpotent.

For the other conjectures, the assumption that $f_0$ be a $K$-system also places some restrictions on $X$.
In particular, any affine $K$-system is weakly mixing, and  Morris proved that the radical of any weakly mixing affine diffeomorphism is nilpotent \cite[Prop. 4.22]{Wit85}. As a consequence, the Lie group $G$ splits as a semidirect product
$G = N \rtimes S$, where $N$ is nilpotent, and $S$ is semisimple.
The projection of $B$ to any compact factor is Zariski dense.

The first cases to consider, then, are when $G$ is either nilpotent or semisimple.  For Conjectures~\ref{conj: lie group rank 1}-\ref{conj: main smooth}, the former case has been considered in \cite{Sven, Sandfeldt2024}, and will also be addressed  in the upcoming work \cite{DSWX}. In the  semisimple case the progress has been made in \cite{ZhijingWang}.

\medskip
\noindent{\bf The nilpotent case.}  We summarize some of the results obtained to date.
\begin{itemize}
\item When $f_0$ is hyperbolic, the 
structural stability of Anosov diffeomorphisms trivially implies that Conjecture~\ref{conj: main top} holds, even without the the additional assumption $\cZ(f)\doteq \Z(f_0)$ on the perturbation $f$.  The main result in \cite{HW14} shows that Conjectures~\ref{conj: lie group rank 1}-\ref{conj: main smooth} also hold true when $G$ is nilpotent and $f_0$ is assumed to be hyperbolic.  
\item 
In \cite{Sven} Sven Sandfeldt considers volume preserving perturbations of irreducible and partially hyperbolic toral automorphisms with a $2-$dimensional isometric center. In this context, Sandfeldt proves the analogue of Conjecture~\ref{conj: lie group rank 1} up to finite index (i.e. the centralizer is a subgroup up to finite index) and the analogue of Conjecture~\ref{conj: main smooth} with the minor extra assumption that no three eigenvalues of the unperturbed automorphism have the same modulus. The analogue of Conjecture~\ref{conj: main smooth variant} is proved for certain toral automorphisms in low dimension, including all irreducible and ergodic automorphisms of $\TT^4$. Further restricting attention to symplectic perturbations, Sandfeldt proves the analogue of Conjecture~\ref{conj: main smooth variant} for all $C^{5}-$small perturbations.
\item Class $\C$: 
Forthcoming work of Damjanovi\'c, Sandfeldt, Wilkinson and Xu \cite{DSWX} defines a class $\C$ of affine $K$-systems on nilmanifolds $X=N/\Gamma$, where the neutral direction of a partially hyperbolic affine diffeomorphism $f_0$ coincides with the center of the nilmanifold $N$ (this includes, for example, Heisenberg manifolds in any dimension) and will show  Conjecture~\ref{conj: lie group rank 1} is within the class $\C$. 

Sandfeldt \cite{Sandfeldt2024} has proved Conjectures~\ref{conj: main smooth variant} and \ref{conj: main smooth} in the special case when $N$ is a $2-$step nilpotent Lie group with a $1-$dimensional derived subgroup, $\dim[N,N] = 1$. This special class of nilmanifolds consists of all Heisenberg nilmanifolds and all products between Heisenberg nilmanifolds and tori. Sandfeldt also proved Conjecture~\ref{conj: lie group rank 1} up to finite index for certain automorphisms on Heisenberg nilmanifolds.
\end{itemize}

\medskip
\noindent{\bf The semisimple case}.  Suppose that $G$ is semisimple.  Since compact Lie groups cannot carry $K$-systems, $G$ must be noncompact.  Then  by passing to its universal cover, we may assume that  $G$ can be split as a direct product of simple Lie groups, with at least one noncompact simple factor.

 Wendy Zhijing Wang \cite{ZhijingWang}  considers the case where $G$ is $\RR$-split, higher rank, and with no compact factors (such as $SL(n,\RR)$, and  $f_0$ is translation by a generic element of the maximal torus (for example a diagonal element of $SL(n,\RR)$ with no two diagonal entries equal). 
She shows that for a $C^1$-small smooth perturbation $f$ of $f_0$, the centralizer $\cZ^\infty(f)$ is a Lie group (establishing Conjecture~\ref{conj: lie group rank 1} in this setting) and proves that if the dimension of $\cZ^\infty(f)$ is sufficiently large, then 
then $f$ is $C^\infty$ conjugate to an affine diffeomorphism. This establishes Conjecture~\ref{conj: main smooth} in this  context and also proves Conjecture~\ref{conj: main smooth variant} for generic translations on simple, non-compact Lie groups of rank 2 (such as $SL(3,\RR)$).

It is also possible to have affine $K$-systems on $G/\Gamma$, where $G$ has compact factors.
Consider $X=G/\Gamma$, where $G=\operatorname{SO}^0(2,1)\times \operatorname{SO}(3)$, $\Gamma\subseteq G$ is an irreducible lattice  (see \cite{Mo15}, Proposition 5.45 for existence), and $f_0$ is the left translation by $(a,e)$ on $X$, where $a$ is a nontrivial diagonal element of $\operatorname{SO}^0(2,1)\cong \SL(2,\RR)$.
Then the hyperbolically generated subgroup $H$ for $f_0$ is $\operatorname{SO}^0(2,1)\times\{e\}$. It follows from the irreducibility of $\Gamma$ that $\overline{H\cdot \Gamma}=G$, which implies that $f_0$ is a $K$-system. 
 Here $\cZ(f_0) \cong \mathbb R\times \operatorname{SO}(3)$, which is a compact extension of $\RR$, and so $\cZ(f_0)$ itself is a rank-1 factor of $\cZ(f_0)$.

 One can also construct higher dimensional versions of this example, for instance for $G=\operatorname{SO}(2,3)^0\times \operatorname{SO}(5)$. On such a quotient $G/\Gamma$, if $f_0$ is the translation by $(a,e)$, with $a$ a nontrivial element of the Cartan subgroup, then the centralizer $\cZ(f)$ contains a  $\mathbb Z^2$ subgroup  whose action is totally $K$.  Therefore this subgroup (and thus $\cZ(f_0)$) has no continuous rank-one factors.  Whether the conjectures hold for $f_0$ and related examples remains an open question.
 

\medskip
\noindent{\bf The general (mixed semisimple and nilpotent) case.} Suppose $G$ splits nontrivially as $G=N\rtimes S$, with $N$ nilpotent and $S$ semisimple, and $X=G/\Gamma$.
Since $G$ admits a $K$-system, one can show that it can't have a nontrivial normal compact semisimple subgroup (such a subgroup would give rise to a factor of entropy $0$). 

One can then reduce to two possibilities:
\begin{itemize}
\item Compact semisimple-by-nilpotent (isometric extension): here $X=G/\Gamma$ is a $C$-bundle over a nilmanifold $Y$, where $C$ is compact semisimple.  The projection of $f_0$ to $Y$ is a $K$-system, and $f_0$ acts as a translation in the fibers. In the case where $f_0$ is Anosov, Theorem A in \cite{DWXtransitive} implies that the conclusions of Conjecture~\ref{conj: lie group rank 1} must hold.

\item Nilpotent-by-semisimple: 
in this case $X=G/\Gamma$ is a nilmanifold $N/\Gamma_N$ bundle over a semisimple quotient $Y=P/\Gamma_P$, and the monodromy of the $P$-action on $X$ is by partially hyperbolic automorphisms.   The projection of $f_0$ to $Y$ is a $K$-system, acting by translation by an element of $P$.
 In recent work, Wang considers this case for $G=\SL(n,\RR)$ and $N=\RR^{m}$.

\end{itemize}






\subsection{Acknowledgements}
We thank Dave Morris for many useful discussions about
centralizers and Lie theory.  We thank Jinpeng An for explaining his work summarized in Appendix~\ref{app sec: twisted conjugacy}.  We are also grateful to Uri Bader, Bassam Fayad, Katie Mann, 
 Sven Sandfeldt and Wendy Wang for useful discussions and corrections.

 This paper is dedicated to Michael Shub, whose paper \cite{BS97} with Jonathan Brezin served as a starting point for this work.

Damjanovi\'c was supported by Swedish Research Council grants VR2019-04641 and VR2023-03596.
Wilkinson  was supported by NSF grants DMS-1900411 and DMS-2154796.
Xu was supported by NSFC 12090010 and 12090015.

\section{Proof of Theorem~\ref{t=mixing ZLip is ZAff}.}\label{ProofB}

We begin with a lemma.

\begin{lemma}\label{l=invariant framing}
    Let $X=G/B$ be a homogeneous manifold, and $f_0 = L_{g_0}\circ \Phi\in \Aff(X)$. Suppose that  $\{v_i\}_{i=1}^n$ is a basis of $\mathfrak{g}=\mathrm{Lie}(G)$ generating $n$ independent right-invariant vector fields on $G$. Then their projections to $G/B$ give a frame of vector fields on $X$. Moreover, with respect to this frame, the derivative of $f_0$ defines a constant matrix cocycle.
\end{lemma} 
\begin{proof}
It is known that each $v\in \mathfrak{g}=\mathrm{Lie}(G)$ generates a right-invariant vector field on $G$ and hence on $X$, whose corresponding flow through $gB\in X$ is of the form
$$
t\mapsto \exp(tv)\cdot gB.
$$
Composing with $f_0$, the orbit of $gB\in X$ is sent to the orbit of  $f_0(gB) =  g_0\Phi(g)B$ under a new flow:
$$
t\mapsto L_{g_0}\circ \Phi(\exp(tv)\cdot gB)
=\exp(t\cdot \mathrm{Ad}(g_0)\circ D_e\Phi(v))\cdot L_{g_0}\circ\Phi(g)B.
$$
Writing $V_v$ for the vector field generating $\exp_{tv}$,
we thus have  $(f_0)_\ast V_v = V_{\mathrm{Ad}(g_0)\circ D_e\Phi}$, which is independent of basepoint.
In particular,  the frame of vector fields generated by $\{v_i\}_{i=1}^n$ determines a constant matrix cocycle over $f_0$.
\end{proof}

We now proceed to the proof.
\begin{proof}[Proof of Theorem~\ref{t=mixing ZLip is ZAff}, in the case where $f_0$ is weakly mixing] Suppose that $f_0 = L_g \Phi\in \Aff(X)$ is weakly mixing.

For convenience, we consider the complexification of tangent spaces and tangent maps in the following proof. This will not affect the affine property of centralizers.

Lemma~\ref{l=invariant framing} gives a basis of vector fields with respect to which $Df_0$ is constant.
By a 
change of coordinates, we may assume that 
 $Df_0$ is in Jordan form with respect to this basis:
\[Df_0 \equiv \begin{pmatrix} \Lambda_{1}& 0& \cdots & 0 \\
0& \Lambda_{2}& \cdots & 0 \\
0& \cdots & \cdots &0 \\
0& 0 & 0 &\Lambda_N \\
\end{pmatrix},
\]
where each $\Lambda_j$ is a Jordan block for $Df_0$ in the minimal generalized eigenspace $E_{\Lambda_j}$ corresponding to an eigenvalue $\lambda_j\in \CC$ (there can be more than one block with the same eigenvalue).

Suppose that  $g \in \Z^{\Lip}(f)$. Then $g, g^{-1}$ are differentiable almost everywhere with bounded derivative. We apply the same  change of coordinates to $Dg$ and write, for a.e.~$x$,
\begin{equation}\label{e=Dg}D_x g =  \begin{pmatrix} A_{11}(x)& A_{12}(x)& \cdots & A_{1N}(x)\\
A_{21}(x)&\cdots & \cdots & A_{2N}(x) \\
\cdots & \cdots & \cdots &\cdots \\
A_{N1}(x) & \cdots & \cdots &A_{NN}(x) \\
\end{pmatrix},
\end{equation}
where the block decomposition is with respect to the minimal generalized eigenspaces for $Df_0$.
Since $g\circ f_0^n = f_0^n\circ g$, we have the almost everywhere equality
\begin{equation}\label{e=matrix commutation}
D_xg = Df_0^{-n}\, D_{f_0^n(x)}g  \, Df_0^{n},\end{equation}
for all $n\in \ZZ$.


Expressing \eqref{e=matrix commutation} in terms of the $A_{ij}$, we obtain the almost everywhere identity
\begin{equation}\label{e=block commutation}\Lambda_i^n \, A_{ij}(x) = A_{ij}(f_0^n(x))\,\Lambda_j^n,
\end{equation}
for all $n\in \ZZ$ and all $1\leq i,j\leq N$. 
As remarked above, since $g$ is Lipschitz,  the norms of the $A_{ij}$ are essentially bounded.






Since each $\Lambda_i$ is a Jordan block, it takes the form
\[\Lambda_i = \Delta_i + P_i,
\]
where $\Delta_i = \lambda_iI_{d_i}$, and 
\[
P_i = \sum_{p=1}^{d_i-1} \delta_{d_i}^{p\,p+1},
\]
for some $d_i\in \{1,\ldots, \dim(X)\}$ (in the case where $d_i=1$, we set $P_i=0$). Here $\delta^{p\,q}_{d_i}$ denotes the $d_i\times d_i$ matrix with all entries $0$, except $\left(\delta^{p\,q}_{d_i}\right)_{p,q}= 1$.
Note that for $k\geq 1$, we have
\begin{eqnarray*}P_i^k & =& \sum_{p=1}^{d_i-k} \delta^{p\,p+ k}_{d_i},
\end{eqnarray*}
which vanishes for $k> d_i-1$,
and so for $n\geq d_i$, we have
\begin{eqnarray*} \Lambda_i^n & =& \Delta_i^n + n  \Delta_i^{n-1} P_i + \cdots + {n \choose {d_i-1}} \Delta_i^{n-(d_i-1)} P_i^{d_i - 1}\\
&=& \Delta_i^n + \sum_{k=1}^{d_i-1} \sum_{p=1}^{d_i-k} {n \choose {k}} \lambda_i^{n-k} \delta_{d_i}^{p\,p+ k}.
\end{eqnarray*}
Note that ${n \choose {k}} \asymp C_k n^k$ as $n\to\infty$ with $k$ fixed.

Now fix $i,j\in\{1,\ldots, N\}$ and write $A(x) = A_{ij}(x)$. To avoid confusion with the block matrices $A_{ij}$, we shall write $a_{p,q}$ for the $p,q$-th entry of $A$.
We claim that the matrix $A$ is constant almost everywhere.

To see this, we first rewrite
the  identity \eqref{e=block commutation} as an a.e.~matrix equation
\begin{align}\label{e=main polynomial equation}
\lambda_{i}^n
U_n A(x)  =  \lambda_j^n A(f_0^n(x)) V_n,
\end{align}
where
\begin{align}
\notag U_n:= I_{d_i}   + \sum_{k=1}^{d_i-1} \sum_{p=1}^{d_i-k} {n \choose {k}} \lambda_i^{-k} \delta^{p\,p+ k}_{d_i},\hbox{ and } V_n:=  I_{d_j}  + \sum_{k=1}^{d_j-1} \sum_{p=1}^{d_j-k} {n \choose {k}} \lambda_j^{-k}\delta_{d_j}^{p\,p+ k}.
\end{align}
Note that for $k$ fixed, the entries on $U_n$ and $V_n$ grow polynomially in $n$ as $|n|\to\infty$, at the  rates $|n|^{d_i-1}$ and $|n|^{d_j-1}$, respectively.

Since $U_n$ is unipotent, it is invertible, and we can rewrite \eqref{e=main polynomial equation} as
\begin{align}\label{e=main polynomial equation2}
\left(\frac{\lambda_{i}}{\lambda_j}\right)^n   A(x)  = U_n^{-1}  A(f_0^n(x)) V_n.
\end{align}
Suppose first that $|\lambda_i| > |\lambda_j|$.
Since the norm of $A$ is $a.e.$ bounded, and the entries of $U_n, V_n^{-1}$ grow at most polynomially in $n$, we see immediately that $A=0$, 
almost everywhere.  Similarly, if $|\lambda_i| < |\lambda_j|$, then  $A=0$, 
almost everywhere.  Thus we may assume that $|\lambda_i| = |\lambda_j|$. 

Let $\omega:= \lambda_j\lambda_i^{-1}$, and rewrite the equation \eqref{e=main polynomial equation} as 
\begin{equation}\label{UA=AV}
U_n A(x) = \omega^n  A(f_0^n(x)) V_n,
\end{equation}
with the equivalent formulation
\begin{equation}\label{UA=AV2}
\omega^n  A(x) V_n =  U_n A(f_0^{-n}(x)).
\end{equation}
We now examine this family (in $n\geq 0$)  of equations entrywise.

For any $d_i\times d_j$ matrix $C$, the product $\delta_{d_i}^{p \,q } C$ is the $d_i\times d_j$ matrix all of whose rows are $0$, except the $p$th row, which is equal to the $q$th row of $C$, and $C \delta_{d_j}^{p\,q}$ is the  $d_i\times d_j$ matrix all of whose columns are $0$, except the $q$th column, which is equal to the $p$th column of $C$.  From this it follows that for $r\in [1,d_i], s\in [1,d_j]$, the $r,s$ entry in the left hand side of 
\eqref{UA=AV} can be expressed in the entries of $A(x)$ in the following functional equation
\begin{align}\notag
\left(U_n A(x)\right)_{r,s} = a_{r,s}(x) + {{n}\choose{1}}\lambda_i^{-1}
a_{r+1,s}(x) + {{n}\choose{2}}\lambda_i^{-2}
a_{r+2,s}(x) + \cdots \\
+ {{n}\choose{{d_i-r}}}\lambda_i^{-(d_i-r)} a_{d_i,s}(x),\label{e=LHS of UA=AV}
\end{align}
and the right hand side may be expressed (for each $n\geq 0$) by
\begin{align}\notag
\left(\omega^n  A(f_0^n(x)) V_n\right)_{r,s} = 
\omega^n \left(a_{r,s}(f_0^n(x)) + {n\choose 1}\lambda_j^{-1}
a_{r,s-1}(f_0^n(x)) + {n\choose 2}\lambda_j^{-2} a_{r,s-2}(f_0^n (x))\right.\\
+ \cdots 
+ \left.{n\choose{s-1}}\lambda_j^{-(s-1)} a_{r,1}(f_0^n (x))\right).\label{e=RHS of UA=AV}
\end{align}
Let $k_0 = \max\{d_i-r, s-1\}$.  Setting equations \eqref{e=LHS of UA=AV} and \eqref{e=RHS of UA=AV} to be equal almost everywhere, and dividing both sides by $n\choose k_0$, we obtain an a.e.~equation in some of the entries of $A$.  Since $|\omega^n|=1$ and $A$ is almost everywhere bounded, this family of equations can be written  (using \eqref{UA=AV2} in case (c)) 
in one of the following forms.
\begin{enumerate}
    \item[(a)] $a_{d_i,s}(x) = \omega^{n-k_0} a_{r_,1}(f_0^n(x)) + O(n^{-1})$, if $k_0 = d_i-r= s-1$,
    \item[(b)] $a_{d_i,s}(x) = O(n^{-1})$, if $d_i-r >s-1$, or
    \item[(c)]  $a_{r,1}(x) = O(n^{-1})$, if $d_i-r < s-1$. 
\end{enumerate}

Clearly if $a_{p,q}(x) = O(n^{-1})$ for all $n$, then $a_{p,q}\equiv 0$ almost everywhere. Thus for the equations in (b) and (c), the leading term (either ${n\choose d_i-r}\lambda_i^{-(d_i-r)}a_{d_i,s}(x)$ in case (b) or $\omega^{n}{n\choose s-1}\lambda_j^{-(s-1)}a_{r,1}(f_0^n(x))$ in (c)) in the $r,s$ entry of equation \eqref{UA=AV} vanishes almost everywhere.

Suppose then that $d_i-r= s-1$ and the equation in (a) holds almost everywhere. The form of this equation may be written abstractly as an almost everywhere equation
\begin{equation}\label{e=abstract iterative relation} b (x) =  \omega^{n} a(f_0^n(x)) + O(n^{-1}),
\end{equation}
where $a,b$ are bounded measurable functions (in this case $a=\omega^{-(s-1)}a_{r,1}$ and $b=a_{d_i,s}$). We use the following lemma.

\begin{lemma}\label{l=functional equation} Let $f_0\in\Aff(X)$, and suppose that $a,b\colon X\to \CC$ are bounded measurable functions satisfying \eqref{e=abstract iterative relation}, for some $\omega\in\CC$ with $|\omega| = 1$, and all $n\in\ZZ$. 

Then for $\mu_{\Haar}$-a.e.\ $x\in X$, we have
\[
b(x) = \omega\,b(f_0(x)), \quad a(x) = \omega\,a(f_0(x)), \quad \text{and } a(x) = b(x).
\]
In particular, if $f_0$ is ergodic and $\omega = 1$, then $a=b$ is constant almost everywhere, and if $f_0$ is weakly mixing, and $\omega\neq 1$ then $a=b=0$ almost everywhere.

\end{lemma}
\begin{proof}
Applying \eqref{e=abstract iterative relation} twice, we have 
\begin{eqnarray*}
\omega\cdot b(f_0(x))& = &\omega^{n+1} a(f_0^{n+1}(x)) + O(n^{-1})\\
&=& b(x) + O(n^{-1}),
\end{eqnarray*}
for almost every $x$.
Since $n$ is arbitrary, we must have $b(f_0(x)) = \omega^{-1}\cdot b(x)$  almost everywhere.

Similarly, by composing with $f_0^{-n}$, we obtain from  \eqref{e=abstract iterative relation} the a.e.~system of equations
\begin{equation*} a (x) =  \omega^{-n} b(f_0^{-n}(x)) + O(n^{-1}),
\end{equation*}
and we conclude that $a(x) = \omega^{-1}a(f_{0}^{-1}(x))$ a.e., which gives that $a(x) = \omega\cdot a(f_0(x))$, almost everywhere.

Using  that $a(f_0^n(x)) = \omega^{-n} a(x)$, we obtain from \eqref{e=abstract iterative relation} that
\[
b(x) =  a(x) + O(n^{-1}),
\]
for all $n$, which implies that $a(x)=b(x)$, a.e.

This proves the first assertions of the lemma.  Since $a,b$ are eigenfunctions of $f_0$ with eigenvalue $\omega$, it is immediate that if $f_0$ is ergodic and $\omega=1$, then $a=b$ are constant, and if $f_0$ is weakly mixing and $\omega\neq 1$, then $a=b=0$, almost everywhere.
\end{proof}

Applying Lemma~\ref{l=functional equation} to the equation in (a), we obtain from the ergodicity of $f_0$ that in the case $\omega=1$, we must have $a_{d_i,s} = a_{r,1}$ are constant almost everywhere.  Restricting to a full measure, $f_0$-invariant subset, the terms  ${n\choose d_i-r}\lambda_i^{-(d_i-r)}a_{d_i,s}(x)$ and $\omega^{n}{n\choose s-1}\lambda_j^{-(s-1)}a_{r,1}(f_0^n(x))$
can then be canceled from both sides in the $r,s$ entry of the functional equation \eqref{UA=AV}. Similarly, since $f_0$ is assumed to be weakly mixing,  if $\omega\neq 1$, then both leading terms vanish.   In either case, we can remove the ``top order" terms (in the coefficients of the matrix entries) from equations \eqref{e=LHS of UA=AV} and \eqref{e=RHS of UA=AV} after setting them equal.

(We remark here that even if $f_0$ is not ergodic, the leading terms can still be canceled, since  Lemma~\ref{l=functional equation} implies that  $a_{d_i,s}(x) = \omega^{-(s-1)} a_{r,1}(x)  =  \omega^{n-(s-1)} a_{r,1}(f_0^n(x))$).

Repeating this argument, setting \eqref{e=LHS of UA=AV} and \eqref{e=RHS of UA=AV} equal almost everywhere and dividing by ${n\choose k_1}$, for an appropriately chosen $k_1$, we remove the next highest order terms from the $r,s$ entry of \eqref{UA=AV}.  

Arguing inductively,  we are eventually reduced to the equation
\[a_{r,s}(x) = \omega a_{r,s}(f_0(x)),\]
and the weak mixing of $f$ then implies that $a_{r,s}$ is constant almost everywhere (and $0$ if $\omega\neq 1$).  Since $r,s$ are arbitrary, we conclude that $A = A_{i,j}$ is constant a.e. 
(and $0$ if $\lambda_i\neq \lambda_j$). 

Since $i,j$ are arbitrary, we have that $Dg$ is constant almost everywhere with respect to the left-invariant basis we started with. This implies that $g$ is affine. This completes the proof of Theorem~\ref{t=mixing ZLip is ZAff}.

We remark that the weak mixing hypothesis was  used only in the case where there are distinct eigenvalues $\lambda_i\neq \lambda_j$ of $Df_0$ with $|\lambda_i| =  |\lambda_j|$; otherwise, ergodicity of $f_0$ suffices. 
\end{proof}

\begin{remark}\label{remark 10}
Our proof of Theorem~\ref{t=mixing ZLip is ZAff} in the weak mixing case actually implies that for every ergodic affine diffeomorphism such that no two distinct eigenvalues have the same modulus, the Lipschitz centralizer is affine. In particular a typical $\RR$-semisimple ergodic homogeneous translation has this property and thus  has affine Lipshitz centralizer. This also generalizes the result  \cite[Theorem C.1]{DSVX}, where it is proved that for any ergodic, $\RR$-semisimple homogeneous flow $\varphi_t$, the Lipschitz and affine centralizer of $\varphi_t$ coincide. To see this, observe that $\RR$-semisimplicity of a homogeneous flow implies all eigenvalues are real, and any ergodic homogeneous flow contains an ergodic element with no negative eigenvalues (thus eliminating 
the case where $\lambda,-\lambda$ are both eigenvalues of $Df_0$).
For a semisimple automorphism $\Phi$, the same result holds, provided that $\Phi^2$ is ergodic.
 \end{remark}


\begin{proof}[Proof of Theorem~\ref{t=mixing ZLip is ZAff} in the ergodic case]
We first record the following corollary of the proof of Theorem~\ref{t=mixing ZLip is ZAff} in the weak mixing case.
\begin{coro}\label{c = Dg classification}
Let $f_0$ be ergodic, and choose a Jordan frame for $T X\otimes \CC$ as above. Denote by $\lambda_i$ the eigenvalue of $Df_0$ corresponding to the block $\Lambda_i$, and let $\omega_{i,j}:= \lambda_i/\lambda_j$. Let $g\in \cZ^{\Lip}(f_0)$.  Then with respect to this frame, $D_xg$ takes the block form \eqref{e=Dg}, where
\begin{enumerate}
    \item $A_{i,j}(x)=0$ if $\omega_{i,j}\neq 1$, and $\omega_{i,j}$ is not an eigenvalue of the operator $U_f: \phi\to \phi\circ f_0$ (in particular, this holds if $|\omega_{i,j}|\neq 1$);\label{e=omega not an eigenvalue}
    \item $A_{i,j}(x)$ is a constant matrix $A_{i,j}$ satisfying $\Lambda_i A_{i,j} =A_{i,j}\Lambda_j $, if $\omega_{i,j}=1$ (in particular, this holds if $i=j$); or \label{e=omega is one}
    \item the nonzero entries of $A_{i,j}(x)$ are  eigenfunctions of $U_f$, with eigenvalue $\omega_{i,j}$, if $\omega_{i,j}$ is an eigenvalue of $U_f$.\label{e=omega is eigenvalue}
\end{enumerate}
\end{coro}
 
Let $\Gg$ be the set of all $C^\infty$ diffeomorphisms $g$ of $X$ such that, when written in block form, the blocks of $D_xg$ and $D_xg^{-1}$  satisfy properties \eqref{e=omega not an eigenvalue}-\eqref{e=omega is eigenvalue}, where we restrict to the eigenfunctions of $U_{f_0}$ on $C^\infty(X,\CC)$.
\begin{lemma}
    $\Gg$ is a closed subgroup of $\Diff^\infty(X)$,  and $\cZ^\infty(f_0)\subseteq \Gg$.
\end{lemma}
\begin{proof} Clearly the identity belongs to $\Gg$, and by definition $\Gg$ is closed under inversion.

Recall that $\omega
_{i,j} = \lambda_i/\lambda_j$.  It follows that if $\phi_{i,j}$ is an eigenfunction of $\omega_{i,j}$ and  $\phi_{j,k}$ is an eigenfunction of $\omega_{j,k}$, then $\phi_{i,j}\cdot \phi_{j,k}$ is an eigenfunction of $\omega_{i,k}$.  The chain rule then implies that $\Gg$ is closed under composition.

Since a $C^\infty$ limit of smooth $\omega_{i,j}$-eigenfunctions for $U_{f_0}$ is a smooth $\omega_{i,j}$-eigenfunction, $\Gg$ is closed.

To check that $\cZ(f_0)\subseteq\Gg$, note that if $g\in \cZ(f_0)$, then $g^{-1}\in \cZ(f_0)$. The inclusion then follows from Corollary~\ref{c = Dg classification}.
\end{proof}
The next lemma is standard.
\begin{lemma} Suppose $f_0\in\Aff(X)$ is ergodic, and there exists  $\phi_0\in L^2(X,\mu_{\Haar})$, with $\phi_0\neq 0$ such that $\phi_0\circ f_0 = \omega \cdot \phi_0$, for some $\omega\in \CC$, $|\omega|=1$.  Then for every $\phi \in L^2(X,\mu_{\Haar})$, if $\phi \circ f_0 = \omega\cdot \phi$, then there exists $c\in\CC$ such that $\phi = c\cdot \phi_0$.
\end{lemma}

\begin{proof} Note first that $|\phi_0|$ is $f_0$-invariant, and so constant, and this constant is not equal to $0$, by assumption.  If $\phi$ also satisfies $\phi \circ f_0 = \omega \cdot \phi$, then $\phi/\phi_0$ is $f_0$-invariant, and hence constant.
\end{proof}

\begin{lemma}
    $\Gg$ is a Lie group.  
\end{lemma}

\begin{proof}
For each $i,j$ such that $\omega_{i,j}$ is an eigenvalue for $U_{f_0}$, either $A_{i,j}(x)\equiv 0$ or there is a nonzero entry, in which case there exists a $C^\infty$ $\omega_{i,j}$-eigenfunction $\phi_{i,j} \colon X \to \CC$ with $|\phi_{i,j}|=1$. Fix these choices $\{\phi_{i,j}\}$, where $i,j$ ranges over the appropriate indices (where $\omega_{i,j}$ is a smooth eigenvalue of $f_0$).

Then the derivative cocycle $Dg$ for each $g\in \Gg$ maps to a  matrix $B_g$ in $\GL(d,\CC)$ by replacing each entry of the form $c\phi_{i,j}(x)$ in the cocycle $D_xg$ with a corresponding entry $c$ in $B_g$.  This defines a continuous embedding of $\Gg$ into a closed subgroup of $\GL(d,\CC)$; hence $\Gg$ is a Lie group.  Since $\Gg$ acts continuously on $X$ by diffeomorphisms, its smooth structure as a Lie group is inherited from the $C^\infty$ topology on $\Diff^\infty(X)$.
\end{proof}

Since $\cZ^\infty(f_0)$ is a closed subgroup of $\Gg$ (in any topology), we obtain the following.
\begin{coro}
 If $f_0\in\Ee(X)$, then  $\cZ^\infty(f_0)$ is a Lie subgroup of $\Diff^{\infty}(X)$; that is:
 \[
 \Ee(X) \subset \Ll(X).
 \]
\end{coro}

The proof of Theorem~\ref{t=mixing ZLip is ZAff} is then reduced to the following result.

\begin{theorem} For every $f_0\in \Ee(X)$, the bi-Lipschitz and smooth centralizer coincide: $\cZ^{\Lip}(f_0) = \cZ^{\infty}(f_0)$.
\end{theorem}

\begin{proof}[Idea of proof]
Let $\Gg_0$ be the set of all bi-Lipschitz homeomorphisms $g$ of $X$ such that, when written in block form, the blocks of $D_xg$ and $D_xg^{-1}$  satisfy properties \eqref{e=omega not an eigenvalue}-\eqref{e=omega is eigenvalue}. (Probably the most appropriate class to consider is the set of homeos with one derivative in $L^2$).
We would then use a theorem like this:

\begin{theorem} Let $f_0\in \Aff(X)$ be ergodic, and let $U_{f_0}\colon L^2(X,\mu_{\Haar})\to  L^2(X,\mu_{\Haar})$ be the associated unitary operator  (whose spectrum lies in $\{w\in\CC: |w|=1\}$).

Suppose that $\omega$ is an eigenvalue of $U_{f_0}$, and let \[E_\omega = \{ \phi\in  L^2(X,\mu_{\Haar}): U_{f_0} \phi = \omega \phi\} .\]
Then there exists $\phi_\omega\in C^\infty(X,\CC)$ such that $|\phi_\omega(x)| = 1$ for all $x\in X$, and
$E_\omega = \CC \cdot \phi_\omega$.

\end{theorem}

\begin{proof}  In \cite[Theorem 1]{Pa69}, it is proved that every eigenfunction of an ergodic affine transformation of a nilmanifold factors through the abelianization. This implies the statement of the theorem for nilmanifolds. 
 Theorem 7.1 in Dani's paper and Theorem 6.2 of Brezin-Moore can then be used to reduce  to the case of a nilmanifold.
   
\end{proof}

We conclude that $\Gg_0 = \Gg$, and we are done.
\end{proof}
This completes the proof of Theorem~\ref{t=mixing ZLip is ZAff} in the ergodic case.
\end{proof}


We complete this section with the discussion of the example due to Walters that demonstrates why a weakly mixing hypothesis on $f_0$ is necessary to conclude that $\cZ^\infty(f_0) \subset \Aff(X)$.

\medskip
\noindent
{\bf Walters's example} \cite{Wa-example}.
Let $\Phi\colon {\TT^n}\to {\TT}^n$ be an ergodic automorphism of the form
$x\mapsto Ax$, where $A\in\SL(n,\ZZ)$ has a complex eigenvalue $\lambda$, $|\lambda| = 1$. (Notice that the lowest possible dimension in which such $\Phi$ exists is $n=4$).  Let $V_\lambda \subset \RR^n$ be the corresponding real eigenspace, with the property that for all $v\in V_\lambda$,
\[Av = R_\theta v,
\]
where  $\theta = \arg(\lambda)$, and $R_\theta$ is the rotation on $V_\lambda$ through angle $\theta$, fixing the complementary generalized eigenspaces.  

Define $f_0\colon \TT^{n+1}\to \TT^{n+1}$ by
\[f_0(x,t) = (\Phi(x), t+\frac{\theta}{2\pi}),\]
and note that $f_0$ is ergodic (though not weakly mixing).
Fix $v\in V_\lambda$, and define $g_v\colon \TT^{n+1}\to \TT^{n+1}$ by
\[g_v(x,t) = (x + R_{2\pi t}v, t). \]
Then $g_v\in \Z^\omega(f)$, but $g_v$ is not affine.  Thus $\Z^\omega(f) > \Z^\omega(\Phi) \times \CC$.\footnote{Walters proves equality here: $\Z^0(f) = \Z^\omega(\Phi) \times \CC$.}  Note that Walters' example has two eigenvalues of the same modulus -- $1$ and $\lambda = e^{\theta i}$ -- which is exactly why $\cZ^\infty(f_0)\not\subset \Aff(X)$(since $f_0$ is not weakly mixing).

Interestingly, $g_v$ {\em is} an affine diffeomorphism if we put a different homogeneous structure on $\TT^{n+1}$ in which $f_0$ remains affine.\footnote{We thank Uri Bader for asking us about this possibility.}  That is, if we write $\TT^{n+1} = G/\Gamma$, where $G$ is a suitably-chosen nonabelian Lie group, then $f_0$ remains affine and $g_v$ becomes affine. In this particular case, we can write  $G= (V\times W) \rtimes \RR$, where
$W$ is the sum of the complementary eigenspaces, and
\[(v,w,t)\cdot (v',w',t') = (v' + R_{2\pi t'} v, w' + w, t'+ t).\] Then  $\Gamma = \ZZ^{n+1}$ (the set if integer points in $G$) is a subgroup of $G$, since $R_{2\pi n} = Id$, for $n\in \ZZ$. Since $\Phi$ is an automorphism of $G$ preserving $\Gamma$, and $(v,w,t)\mapsto (v,w,t+\theta/2)$ is a translation in $G$, we conclude that $f_0\in\Aff(G/\Gamma)$.  In this new homogeneous structure, $g_v$ is also affine: it is just translation by $(v,0,0)$.

Perhaps every ergodic affine map has affine centralizer with respect to some ``maximal" compatible homogeneous structure, and weakly mixing just forces this maximal structure to be unique. 
 \medskip
 \noindent
\section{Proof of Theorem~\ref{thm: not K imp big cent}: the case when $G$ is solvable, compact semisimple, or semisimple without compact factors}

 

\subsection{Examples}

In the outline of the proof of Theorem~\ref{thm: liecent iff ergodic} given in Section~\ref{ss=structure}, if $f_0$ is not a $K$-system, then there is an affine perturbation $f$ that takes a special form of a bundle map, and 
there are two cases that can occur, leading to two different constructions of diffeomorphisms in $\cZ(f)$. The following simple examples on $\TT^3$ illustrate the construction of large centralizer in these two cases.

The first example corresponds to Case 1.
\begin{example}\label{ex=torus1}
     Let $f\in\Aff(\TT^3)$ be the (non-ergodic) map defined by
    \[f(x,y,z) = (x+1/2, x+ y, 2x + y + z).\]
 In this example, $f$ preserves the structure of $\TT^3$ as a trivial $\TT^2$-bundle over $\TT$.  The map thus takes the form $(x,(y,z))\mapsto (x+1/2, g_x(y,z))$, where $g_x\in \Aff(\TT^2)$ is given by
 \[g_x(y,z) = B(y,z) + (x,2x),\]
 with $B = \begin{pmatrix}
     1&0\\
     1& 1
 \end{pmatrix}$.
 The key property of $g_x$ we will use here is that $B$ has $1$ as an eigenvalue, and in particular $g_x(y,z+\alpha) = g_x(y,z)+(0,\alpha)$, for all $\alpha\in \RR$ 
 
 The base map $R_{1/2}: x\mapsto x+1/2$ is periodic.  Let $\phi\colon \TT\to \RR$ be an arbitrary smooth function that is $1/2$- periodic, and
define a smooth map $h\colon \TT^3\to \TT^3$ by
\[
h(x,y,z) = (x, y, z+\phi(x)).
\]
By its construction, $h\in\cZ(f)$.  Thus $\cZ(f)$ contains the additive group of all $1/2$-periodic functions in $C^\infty(\TT,\RR)$ (which is the pullback of $C^\infty(\TT,\RR)$ under the covering map $x\mapsto 2x$ on $\TT$ and contains $C^\infty_c(\TT)$).

Notice that ``$1/2$" may be replaced by any rational number in this example, and $B$ can be replaced by any automorphism of $\TT^2$ with $1$ as an eigenvalue.  Thus any affine map of the form 
$(x,(y,z))\mapsto (x+\alpha, g_x(y,z))$, where the linear part of $g_x$ has $1$ an an eigenvalue can be approximated by an affine map whose centralizer contains $C^\infty_c(I)$.
\end{example}

The next example illustrates the situation in Case 2.

\begin{example}\label{ex=torus2}
    Let $f\in\Aff(\TT^3)$ be the (nonergodic) map defined by
    \[f(x,y,z) = (x+1/2, x+ 2y + z, y+z).\]
    In this example, $f$ also preserves the $\TT^2$-bundle structure, and so $(x,(y,z))\mapsto (R_{1/2}(x), g_x(y,z))$, this time with
\[g_x(y,z) = B(y,z) + (x,0),\]
where $B = \begin{pmatrix}
     2&1\\
     1& 1
 \end{pmatrix}$.
In this case, $B$ does {\em not} commute with any affine vector field, as it does not have $1$ as an eigenvalue.  On the other hand, the fact that $B-I$ is invertible implies we can find a smooth  coordinate change on $\TT^3$ so that $g_x$ becomes a product $R_{1/2}\times B$. Then any diffeomorphism of the form $\psi\times \Id$, where $\psi$ commutes with $R_{1/2}$, centralizes $f$. Since $\cZ(R_{1/2})$ contains $\Diff^\infty_c(I)$, so does $\cZ(f)$.  
\end{example}

These two examples might seem special, but in fact any non-stably ergodic affine map of $\TT^3$ can, after an affine perturbation and change of coordinates, be put in the form of a skew product over a periodic translation.  Then one of the two methods above can be used to build a huge non-affine centralizer for this perturbation.  The details for higher dimensional tori are similar, although the change of coordinates and perturbation are subtler.   The argument is detailed in Section~\ref{sec: abelian case}.


For other Lie groups, the final step of the argument (building the huge centralizer) follows the same type of method (see the ``key propositions" in Section~\ref{sec: key propositions}). 

We proceed to the proof of Theorem~\ref{thm: not K imp big cent}.

\subsection{Key propositions}\label{sec: key propositions}
In our proof, we will work with  smooth fiber bundles with fibers modeled on homogeneous spaces. On such bundles, we will focus on a specific class of bundle morphisms: those that  are affine along fibers (with constant linear part) and that project to  a periodic map on the base. The {\em linear part} of an affine map $f:=L_a\circ \Phi$ on a homogeneous space $G/B$ is the automorphism  $\Phi\in \mathrm{Aut}(G,B)$.
We now define these notions more precisely.

\begin{definition}\label{def: homo fib}
A {\em homogeneous fiber bundle} is a smooth fiber bundle  $F\to X\xrightarrow{\pi} \bar{X}$ such that the fiber $F$ is modeled on a homogeneous space $G_F/B_F$, where $G_F$ is a connected Lie group and $B_F$ is a closed cocompact subgroup of $G_F$. 
\end{definition}
Equivalently, $X$ is a homogeneous fiber bundle if admits a transitive, locally free action of a connected Lie group $G_F$ projecting to the trivial action on $\bar X$.
Examples of homogeneous fiber bundles include principal bundles (with compact structure groups) and  unit tangent bundle of Riemannian manifolds. 

\begin{definition}\label{def: special morphism}
 A smooth bundle morphism $f$ of a homogeneous fiber bundle $F\to X\xrightarrow{\pi} \bar{X}$ is called \textit{special}   if  
\begin{enumerate}
\item $\dim \bar X\geq 1$, and the induced base map $\bar{f}$ on $\bar X$ is a $d$-periodic diffeomorphism on $\bar X$ for some (minimal) positive integer $d$.
\item there exists an open-dense subset of $\bar X$ in which each point $x$ admits a  neighborhood $U_x$ such that
\begin{itemize}
    \item[-] ${U_x},\bar{f}({U_x}),\cdots,\bar{f}^{d-1}({U_x})$ are pairwise disjoint, and
    \item[-] there exists a trivialization of $\pi$ over $\bigsqcup_{j=0}^{d-1} \bar{f}^{j}({U_x})$ such that, in these coordinates, $f$ is affine along fibers and the linear part of $f$ is an automorphism $A\in \mathrm{Aut}(G_F,B_F)$ that doesn't depend on the base point. 
    \end{itemize}
 \end{enumerate}
\end{definition}

In (2) of Definition \ref{def: special morphism}, we call the automorphism $A$ the {\em linear part of $f$ along fibers}.  Note that $f$ may have different linear parts  along fibers, depending on the choice of $x$ and local trivialization.

%

The following proposition corresponds to Case 1 of the proof outline in Section~\ref{ss=structure}, as illustrated in Example \ref{ex=torus1}.

\begin{prop}[Key proposition, part 1]\label{prop: key prop 1} Let $F\to X\xrightarrow{\pi} \bar{X}$ be a homogeneous fiber bundle, 
and  let $f$ be a special bundle morphism of $X$.
Assume that
$1$ is an eigenvalue of $D(A^d)$ where $A$ is a linear part of $f$ along fibers, and $d$ is the period of the base map $\bar f$. 

Then there exist an $\bar f$-invariant, nonempty open set $\Vv\subset \bar X$, and a
$Df$-invariant,  affine  vector field $V$ on $X$ that is tangent to the fibers of $\pi$ and nonvanishing on $\pi^{-1}(\Vv)$.  Consequently, the smooth centralizer of $f$ contains $C^\infty_c(I)$, and thus is not a finite dimensional Lie group.
\end{prop}

\begin{proof} Let $F\to X\xrightarrow{\pi} \bar{X}$ be a homogeneous fiber bundle with $F = G_F/B_F$, and let $f$ be the special bundle morphism, with base period $d$.

We may assume that  $G_F$ is a minimal (with respect to inclusion) connected Lie group acting transitively on the fibers of $X$, and that the point stabilizer $B_F$ contains no non-discrete closed subgroups $B'$ that are normal in $G_F$ (see Remark \ref{reducetononnormal} in Appendix~\ref{app: A} for details). 

Since $f$ is a special morphism, there exists  $U\subseteq \bar{X}$ such that $U\cap \bar{f}^j(U)=\emptyset, 1\leq j\leq d-1$ and 
and a local trivialization from $\pi^{-1}(\bigsqcup_{j=0}^{d-1}\bar{f}^j(U))$ to $(\bigsqcup_{j=0}^{d-1}\bar{f}^j(U))\times F$ such that 
under this identification, $f$ takes the form
$(x,y)\mapsto (\bar{f}(x),a_{x}A(y))$,
    where $a\colon \bigcup_{j=0}^{d-1}\bar{f}^j(U)\to G_F$ is a smooth function
    and $A \in\mathrm{Aut}(F)$.  By assumption,  $1$ is an eigenvalue of $D(A^d)$.

Shrinking $U$ if necessary, 
 we first construct a nonvanishing $Df^d$-invariant vector field $V_0$  on $\pi^{-1}(U)$ tangent to the $\pi$-fibers.  This will extend to a $Df$-invariant vector field $V_0$ on $\pi^{-1} \left(\bigsqcup_{j=0}^{d-1}\bar{f}^j(U)\right)$ via pushforward by $Df^j$ on $\pi^{-1}\left(\bar{f}^j(U) \right)$. Fixing a smaller neighborhood $U_0\subset \bar X$ with $\overline{U_0}\subset U$, and using an $\bar f$-invariant bump function $\rho\colon \bigsqcup_{j=0}^{d-1}\bar{f}^j(U) \to [0,1]$ equal to 1 on $\mathcal{V}:= \bigsqcup_{j=0}^{d-1}\bar{f}^j(U)$, we then define  a vector field $V$ on $X$ by
 \[V(p) := \begin{cases} \rho(\pi(v)) V_0(p), & \hbox{if $p\in \pi^{-1}\left(\bigsqcup_{j=0}^{d-1}\bar{f}^j(U)\right)$},\\
 0&\hbox{otherwise.}
 \end{cases}\] This gives the desired $Df_0$-invariant vector field $V$ on $X$.

Moving on to our construction of the vector field $V_0$ on  $\pi^{-1}(U)$, we first write $f^d$ in our trivializing coordinates in $\pi^{-1}(U)$ as
$$
(x,y)\mapsto (x,a_{d,x}\cdot A^d(y)),
$$
for $x\in U$ and $y\in G_F/B_F$,
where $a_{d,x}:=a_{\bar{f}^{d-1}(x)}\cdot A(a_{\bar{f}^{d-2}(x)})\cdots A^{d-1}(a_x)\in G_F$.

Fix $x\in U$. To define a vector field $V_{0,x}$ on $F_x = \pi^{-1}(x)\cong \{x\}\times G_F/B_F$, we will find a $R_{G_F}$-invariant vector field on $G_F$ that is also invariant under $f^d_\ast$ (which will then descend to the quotient $G_F/B_F$). Recall that the Lie algebra $T_0 G_F =\mathfrak{g}_F$ is identified with right-invariant vector fields on $G_F$ via
$v\in \mathfrak{g}_F \mapsto V_v$, where for $g\in G_F$,
$V_v(g) := D_eL_g(v)$, and the flow $\psi^v_t$ generated by this $V_v$
is $\psi^v_t(g) := \exp(tv) g$.  

Thus the condition that $V_{0,x} := V_{v_x}$ be $Df^d$-invariant is equivalent to the condition that $f^d(\psi^{v_x}_t(g)) = f^d (\exp(tv) g) = \exp(tv) f^d(g) = \psi^{v_x}_t(f^d(g))$ hold for all $t\in \RR$ and $g\in G_F$. Using our trivialization of $f$ on $F_x$, this is equivalent to the equation
\begin{equation}\label{commutingflow}
(a_{d,x}A^d)\circ L_{\exp(v_x)}
=L_{\exp(v_x)}\circ (a_{d,x}A^d).
\end{equation}

Thus to find a $Df^d$-invariant vector field $V_{0,x}$ on $\pi^{-1}(x)$, it suffices to solve
 the following equation for $v_x\in \mathfrak{g}_F$:
\[a_{d,x}\cdot A^d(\exp(v_x))\cdot a_{d,x}^{-1}
=\exp(v_x).\]
This Lie group equation then induces the Lie algebra equation
\[
(\mathrm{Ad}_{a_{d,x}}\circ D(A^d))(v_x)=v_x,
\]
i.e. $v_x\in\ker(\mathrm{Ad}_{a_{d,x}}\circ D(A^d)-\mathrm{id})$.  
To solve this equation, we start
with the following lemma from Lie theory:
\begin{lemma}\label{eigen1trans}    Let $G$ be a connected Lie group, and let $A_1,A_2$ be two automorphisms of $G$ such that $A_1\circ A_2^{-1}$ is an inner automorphism. Suppose that $DA_1$ has eigenvalue $1$. Then $DA_2$ also has eigenvalue $1$.\end{lemma}
\begin{proof}See Appendix \ref{app sec: twisted conjugacy}.
\end{proof}

Since $\ker(DA^d-\id)$ is nontrivial, Lemma \ref{eigen1trans} implies that
$\ker(\mathrm{Ad}_{a_{d,x}}\circ D(A^d)-\mathrm{id})\neq \{0\}$.
We may thus choose a unit vector 
$v_x\in \ker(\mathrm{Ad}_{a_{d,x}}\circ D(A^d)-\mathrm{id})$
for each $x\in U$.

This assignment $x\mapsto v_x$ can be chosen to be smooth in $x\in U$.  To see this, note first that the semi-continuity of the rank function implies that
by adjusting $U$ if necessary,
we may assume that $\mathrm{rank} (\ker(\mathrm{Ad}_{a_{d,x}}\circ D(A^d)-\mathrm{id}))$ is constant on $U$.

The rank theorem then gives that there is a local trivialization of $x\mapsto  \ker(\mathrm{Ad}_{a_{d,x}}\circ D(A^d)-\mathrm{id}) $ taking values in a suitable Grassmannian bundle (shrinking $U$ further if necessary). Selecting a constant section of this bundle then gives a smooth function  $x\mapsto v_x\in  \ker(\mathrm{Ad}_{a_{d,x}}\circ D(A^d)-\mathrm{id})$.

Having chosen this assignment,  we define the smooth vector field $V_0$ on $\pi^{-1}(U)$ by $V_0(x,g B_F) = D_eL_g(v_x)\in T_g (G_F/B_F)$.
 Then $V_0$ generates a smooth flow on $\pi^{-1}(U)$ commuting with $f^d$ by \eqref{commutingflow}.
 Since $B_F$ contains no nontrivial non-discrete closed subgroup that is normal in $G_F$, this flow is nontrivial (see Lemma \ref{nontrivialvectorfield} in Appendix~\ref{app: A}). 
Pushing forward by $Df$ we get a $Df$-invariant vector field $V_0$ on $\pi^{-1}(\bigcup^{d-1}_{j=0}\bar f^j(U))\cong (\bigcup^{d-1}_{j=0}\bar f^j(U))\times F$, and using a bump function as described above gives the desired vector field $V$.

To see that $C_c^\infty(I)$ embeds in $\cZ^\infty(f)$, we note that any
$f$-invariant smooth function $\alpha$ on $X$ that is constant on fibers defines a vector field
$\alpha V$ generating a flow commuting with $f$. 
Embedding the interval $I$ smoothly into the neighborhood $U_0\subset \bar X$, and fixing a bump function in $U_0$ equal to $1$ on the image of $I$, 
we note that the push-forward of any $\phi\in C^\infty_c(I)$ to $U_0$ then extends canonically to a smooth, $\bar f$-invariant function on $\bar X$, giving rise to an $f$-invariant flow $\psi^\phi_t$ on $X$.  The map $\phi\mapsto \psi^\phi_1$ is an injective homomorphism from $C^\infty_c(I)$ (under addition)  to $\cZ^\infty(f)$ that is easily seen to be continuous.

This completes the proof of
Proposition \ref{prop: key prop 1}.\end{proof}

The next key proposition corresponds to Case 2 of the proof outline in Section~\ref{ss=structure}, as illustrated in Example \ref{ex=torus2}.

\begin{prop}[Key proposition part 2]\label{prop: key prop 2}  Let $F\to X\xrightarrow{\pi} \bar{X}$ be a homogeneous fiber bundle,
and  let $f$ be a special bundle morphism of $X$.
Assume that
$1$ is not an eigenvalue of $D(A^d)$ where $A$ is a linear part of $f$ along fibers, and $d$ is the period of the base map $\bar f$. 

Then there exists an  open set $U\subset \bar{X}$ such that
\begin{enumerate}
    \item for $1\leq j\leq d-1$, $\bar{f}^j(U)\cap U=\emptyset$;
    \item there is a smooth coordinate change $$\Psi: \pi^{-1}(\cup_{j=0}^{d-1} \bar{f}^j(U))\to (\cup_{j=0}^{d-1} \bar{f}^j(U))\times F$$ such that $\Psi \circ f \circ \Psi^{-1}=\bar{f}\times A$.
\end{enumerate}
In particular, any diffeomorphism $\bar{g}$ supported in $\cup_{j=0}^{d-1} \bar{f}^j(U))$ and commuting with $\bar f$ induces a diffeomorphism $\bar{g}\times id$ commuting with $\bar{f}\times A$, which extends by the identity to a diffeomorphism commuting with $f$.  Consequently, $\Diff^\infty_c(I)$ embeds continuously into $\cZ^\infty(f)$.
\end{prop}
\begin{proof}As the proof of part 1, we choose a small open neighborhood $U\subseteq \bar{X}$ 
such that $U\cap \bar{f}^j(U)=\emptyset, 1\leq j\leq d-1$ and a local trivialization from $\pi^{-1}(\bigcup_{j=0}^{d-1}\bar{f}^j(U))$ to $(\bigcup_{j=0}^{j-1}\bar{f}^j(U))\times F$ such that under this local trivilization, $f$ is written as $$(x,y)\mapsto (\bar{f}(x), a_xA(y)), a\in C^\infty((\bigcup_{j=0}^{j-1}\bar{f}^j(U)), G_F), A\in \mathrm{Aut}(F)$$
and $1$ is not an eigenvalue of $D(A^d)$.

As in the proof of Proposition~\ref{prop: key prop 1}, we write $f^d$ in coordinates over $U$ as
$$
(x,y)\mapsto (x,a_{d,x}\cdot A^d), \text{ where }a_{d,x}:=a_{\bar{f}^{d-1}(x)}\cdot A(a_{\bar{f}^{d-2}(x)})\cdots A^{d-1}(a_x)\in G_F.
$$

Now fix $x\in U$.
In these coordinates, our desired coordinate change $\Psi_x$ on $F_x = \{x\}\times G_F/B_F$ will take the form
$\Psi_x(x,y) =(x, L_{g_x}(y))$, for some $g_x\in G_F$.  The equation $f^d\circ L_{g_x} = L_{g_x} \circ f^d$ on $F_x$ then becomes
\begin{equation}\label{conjugatinglinear}
(a_{d,x}A^d)\circ L_{g_x}
=L_{g_x}\circ A^d, 
\end{equation}
and the corresponding Lie group equation on $G_F$ is
$a_{d,x} A^d(g_x)
=g_x.$
The existence of $g_x$ is guaranteed by the following lemma; for the proof see Appendix \ref{app sec: twisted conjugacy}:
\begin{lemma}\label{eigen1fix}
Let $G$ be a connected Lie group, and let $A$ be an automorphism of $G$ such that $1$ is not an eigenvalue of $DA$. Then for every $a\in G$, the map 
\[
f\colon G\to G,\qquad f(g)=aA(g)
\]
has a fixed point.
\end{lemma}
As in the proof of Proposition~\ref{prop: key prop 1}, the rank theorem implies that by shrinking $U$ if necessary, there is a smooth choice of assignment  $x\mapsto g_x$ for $x\in U$.
Then we inductively define $g_{\bar{f}^j(x)}:=a_{\bar{f}^{j-1}(x)}\cdot A(g_{\bar{f}^{j-1}(x)})
(1\leq j\leq d-1), x\in U$.
Finally, we set the smooth coordinate change $\Phi$ to be 
$(x,y)\mapsto (x,g_x^{-1}y)$ on $\pi^{-1}(\cup_{j=0}^{d-1} \bar{f}^j(U))$,
which clearly satisfies 
$\Phi \circ f \circ \Phi^{-1}=\bar{f}\times A$.

Let $k = \dim(U) =\dim(\bar X)\geq 1$.  To  embed $\Diff^\infty_c(I^k)$ into $\cZ^\infty(f)$, we smoothly embed $I^k$ in $U$ 
Then any diffeomorphism  $\phi \in \Diff^\infty_c(I)$ is conjugated by this embedding to a compactly-supported diffeomorphism of $U$, which extends canonically via conjugation by $\bar f^j$, $j=1,\ldots, d-1$ to diffeomorphism $\bar g_\phi$ commuting with $\bar f$ in $\bigsqcup_{j=0}^{d-1} \bar{f}^j(U)$. The diffeomorphism $\Psi^{-1}\circ (\bar g_\phi\times \id_{F}) \circ \Psi$ then commutes with $f$.

By extending this diffeomorphism via the identity on $X$, one obtains a diffeomorphism $g_\phi$ in $\cZ^\infty(f)$. The assignment $\phi\mapsto g_\phi$ is a continuous, injective homomorphism.
Since $\Diff_c(I)$ embeds continuously in $\Diff_c(I^k)$, for any $k\geq 1$\footnote{To embed $\Diff_c(I^j)$ into $\Diff_c(I^{j+1})$ for $j\geq 1$,   first embed  $\Diff_c(I^j)$ into $\Diff_c(I^j\times \TT)$ 
via $\phi\mapsto \phi\times \id_{\TT}$, and then  embed $I^j\times \TT$ into $I^{j+1}$.}, 
This completes the proof.
\end{proof}

\subsection{Building and collapsing  towers of compact homogeneous spaces}\label{sec: tower str}
In our outline of the proof of Theorem~\ref{thm: liecent iff ergodic} given in Section~\ref{ss=structure},  we first construct an $f$-invariant homogeneous fiber bundle structure on $X$ on which $f$ becomes a special morphism, satisfying properties (1) and (2).  In this subsection, we describe techniques for constructing such a fiber bundle.

Our starting point is the following lemma which is standard in geometry (for a proof see Appendix~\ref{app: A}).
\begin{lemma}\label{def: fib seq homo}
Let $X=G/B$ be a compact homogeneous space, and let  $G'$ be a closed subgroup of $G$ such that $G'\cdot B$ is also a closed subgroup of $G$. Then $G/(G'\cdot B)$ is a compact homogeneous space and we obtain a {\em fiber bundle  of homogeneous spaces} $$G'/(G'\cap B)\to G/B\to G/(G'\cdot B), $$ such that any affine map preserving $B$ (resp. left translation) of $X$ projects to an affine map (resp. left translation) on $G/(G'\cdot B)$.
\end{lemma}
In our arguments, starting with $f_0\in \Aff(G/B)$ that is not a $K$-system, and setting $G_0=G, B_0=B$ we will use this lemma iteratively to ``factor out" certain closed subgroups from the bases $G_{i+1}/B_{i+1}$ of a sequence of homogeneous fibrations 
\[G_i'/(G_i'\cap B_i) \to G_{i}/B_{i} \to G_{i+1}/B_{i+1},\, i=0,\ldots
\]
with affine bundle morphisms $f_i\colon G_i/B_i\to G_i/B_i$, until we obtain a map $f_k\colon G_k/B_k\to G_k/B_k$ on a nontrivial quotient that can be perturbed via a left translation to a periodic map $\bar f \in \Aff(G_k/B_k)$.  We will then lift $\bar f$ to an affine perturbation $f$ of $f_0$ and construct a new $f$-invariant homogeneous fibration $G_F/B_F\to X\to \bar X$, where $\bar X:= G_k/B_k$.  In the final step we show that $f$ is a special morphism of this homogeneous fiber bundle structure.

Another technique we will use is to ``remove" a closed subgroup $J < G'\cdot B$ that is normal in $G$. Denoting  by $\pi_J: G\to G/J$ the group homomorphism modulo the common closed normal subgroup $J$, and writing $G_J=\pi_J(G)$, and $B_J=\pi_1(G'\cdot B)$, we obtain the isomorphism between homogeneous spaces 
\begin{equation}\label{eq: fibr str homo} G/(G'\cdot B)\cong G_J/B_J.\end{equation}
By choosing $J$ to be a {\em maximal} closed subgroup of $G'\cdot B$ that is normal in $G$ (which is unique by Lemma~\ref{normalinadmissible} in Appendix~\ref{app: A}), and 
expressing the quotient in this form, we can always assume that that $B_J$ contains no nontrivial closed normal subgroups in $G_J$.




Thus to obtain  a sequence of fibrations
$$G_i'/(G_i'\cap B_i)\to G_i/ B_i\to G_i/(G_i'\cdot B)
\cong G_{i+1}/B_{i+1},$$
we will, at each step $i$ choose an $f_0$-invariant closed subgroup
 $G_i'< G_i$ such that  $G_i'\cdot B_i$ is closed in $G_i$ and a (potentially trivial) maximal closed subgroup $J_i< G_i'\cdot B_i$ that is normal in $G_i$, to obtain $G_{i+1}, B_{i+1}$ via:
 $G_{i+1} = \pi_{J_i}(G_i)$ and $B_{i+1} = \pi_{J_i}(B_i)$.
This process produces a {\em tower of fibrations} by
compact homogeneous spaces with multiple layers: 
\begin{equation}\label{digr: collp}
\begin{tikzcd}
  G'/(G'\cap B) \arrow[r, ] &  G/B \arrow[d,] \\
   G_1'/G_1'\cap B_1 \arrow[r, ] & G_1/B_1 \arrow[d, ] \\
  \dots \arrow[r, ] & \dots\arrow[d, ] \\
  G_{n-1}'/(G_{n-1}'\cap B_{n-1}) \arrow[r, ] &G_{n-1}/B_{n-1}\arrow[d,]\\
  & G_n/B_n, 
\end{tikzcd}
\end{equation}
where each $B_n$ contains no nontrivial closed, normal subgroups of $G_n$. 

\bigskip

\noindent 
{\bf Tower collapse.}
Any tower of fibrations of the form  \eqref{digr: collp}  can be {\em collapsed} to a single homogeneous fiber bundle, by collapsing consecutive rows iteratively.  

For example, given a tower with two layers
$$
\begin{tikzcd}
  G'/(G'\cap B)  \arrow[r,] & G/B \arrow[d, 
  ] \\
  G_1'/(G_1'\cap B_1)  \arrow[r,] & G_1/B_1 (\cong G/(G'\cdot B)) \arrow[d, 
  ] \\
& G_2/B_2(\cong G_1/(G_1'\cdot B_1)\cong G/\pi_{1}^{-1}(G_1'\cdot B_1)),
\end{tikzcd}
$$
where $\pi_1 \colon  G\to G_1$ is the group projection, 
we may form a new homogeneous fiber bundle (in the sense of Lemma \ref{def: fib seq homo}) with \textit{bigger} fiber:
$$\pi_1^{-1}(G_1'\cdot B_1)/(\pi_1^{-1}(G_1'\cdot B_1)\cap B)\to G/B\to G/\pi_1^{-1}(G_1'\cdot B_1).$$
It is easy to check that $ \pi_1^{-1}(G_1'\cdot B_1)$ is a closed subgroup of $G$.

 By a similar argument we can collapse the tower \eqref{digr: collp} to a fiber bundle structure on $G/B$ with \textit{huge} fiber $\pi^{-1} (G_{n-1}'\cdot B_{n-1})/(\pi^{-1} (G_{n-1}'\cdot B_{n-1})\cap B)$, 
$$\pi^{-1} (G_{n-1}'\cdot B_{n-1})/(\pi^{-1} (G_{n-1}'\cdot B_{n-1})\cap B)\to G/B\to G/\pi^{-1}(G_{n-1}'\cdot B_{n-1})\cong G_n/B_n,$$
where $\pi=\pi_{n-1}\circ\cdots \pi_1$ and $ \pi_i:G_{i-1}\to G_i$ is the group projection. In this way, we obtain from tower a single homogeneous fiber bundle in the sense of Definition \ref{def: homo fib}. We will refer to this identification between the tower with multiple layers and the homogeneous fiber bundle with huge fiber $\pi^{-1} (G_{n-1}'\cdot B_{n-1})/(\pi^{-1} (G_{n-1}'\cdot B_{n-1})\cap B)$, as \textit{the collapsing process.}   

The following lemma will allow us to use the collapsing process to apply the key propositions. For an affine map $f: G/B\to G/B$, we say $f$ {\em preserves a fiber bundle tower} if there is a fiber bundle tower as above that is $f$-invariant; $f$ is said to project to a periodic map if the map induced by $f$ on the base space $G_n/B_n$ is periodic.

\begin{lemma}\label{lem: cllpse proc imp spcl fb}Let $f$ be an affine map of a compact homogeneous space that preserves a fiber bundle tower and projects to a periodic map. 
Then after a collapsing process,  $f$ induces a special bundle morphism of the fiber bundle (with huge fiber) $$\pi^{-1} (G_{n-1}'\cdot B_{n-1})/(\pi^{-1} (G_{n-1}'\cdot B_{n-1})\cap B)\to G/B\to G_n/B_n,$$
in the sense of Definition \ref{def: special morphism}.   
\end{lemma}
\begin{proof}See Section~\ref{app: Lemma 18}  in Appendix~\ref{app: A}.
\end{proof}
\begin{coro}\label{coro: 19+13+15 package}Let $f$ be an affine map of a compact homogeneous space that preserves a fiber bundle tower and projects to a periodic map. 
Then $\Z^\infty(f)$ contains a subgroup isomorphic to either $C^\infty_c(I)$ or $\Diff^\infty_c(I)$.
\end{coro}
\begin{proof}By Lemma \ref{lem: cllpse proc imp spcl fb}, $f$ induces a special bundle morphism. Then it follows from Propositions \ref{prop: key prop 1} and \ref{prop: key prop 2} that the smooth centralizer of a special bundle morphism contains $C^\infty_c(I)$ or $\Diff^\infty_c(I)$. 
\end{proof}

\bigskip
\noindent{\bf Translational perturbations and lifting.} 
To review our strategy, starting with $f_0\in \Aff(G/B)$ that is not a $K$-system, we will construct an $f_0$-invariant tower of homogeneous fiber bundles so that the induced map $\bar f_0$ on the base of the tower can be approximated by a periodic affine diffeomorphism $\bar f$.  To ensure that $\bar f$ lifts to a perturbation $\bar f$ of $f_0$ on the total space $G/B$, we will need to restrict to a specific class of perturbations of $\bar f_0$.

\begin{definition}
    A {\em translational perturbation} of $f_0\in \Aff(G/B)$ (of size $\epsilon$) is a map $f = L_{a_\epsilon}\circ f_0\in \Aff(G/B)$, where $a_\epsilon$ lies in an $\epsilon$-neighborhood of $e\in G$.  

     When we say that $f_0$ has a translational perturbation with a given property, we mean that for all $\epsilon>0$ there is a translational perturbation of size $\epsilon$ with that property.
\end{definition}

Translational perturbations have a nice lifting property, given by the following lemma.
\begin{lemma}\label{lem: lift0}
Let $f_0\in\Aff(G/B)$, and let
\[
G_F/B_F \to G/B \to \bar G/\bar B
\]
be an $f_0$-invariant fiber bundle of homogeneous spaces. Denote by $\bar f_0\in \Aff(\bar G/\bar B)$ the induced map on the base.  

Then any translational perturbation $f$ of $f_0$ induces a translational perturbation $\bar f$ of $\bar f_0$. Moreover, for any translational perturbation $\bar f$ of $\bar f_0$, there exists a translational perturbation $f$ of $f_0$ that projects to $\bar f$ on $\bar G/\bar B$.
\end{lemma}

\begin{proof}
Let $\pi\colon G\to \bar G$ be the Lie group homomorphism inducing the given bundle structure.   Clearly if $L_a\circ f_0$ is a translational perturbation of $f_0$ preserving the bundle, then $L_{\bar a}\circ \bar f_0$ is the induced translational perturbation of $\bar f_0$, where $\bar a = \pi(a)$.

On the other hand, $\pi$ is both a submersion and an  open map, and so for any $\bar a$ in a neighborhood of $e\in \bar G$, there exists $a$ in a neighborhood of $e\in G$ such that $\pi(a) = \bar a$.  If $\bar f = L_{\bar a}\circ \bar f_0$, then setting $f= L_a\circ f_0$ gives the desired lift.  
\end{proof}


Combining  Corollary \ref{coro: 19+13+15 package} and Lemmas~\ref{lem: cllpse proc imp spcl fb} and \ref{lem: lift0}, one easily obtains the following.

\begin{coro}\label{coro: combine strategy}
Let $X=G/B$ be a compact homogeneous space, and let $f_0$ be an affine map that preserves a fiber bundle tower of $X$, projecting to an affine map $g_0$ on its base $Y$. 

Suppose  that there exists a translational perturbation $g$ of $g_0$ preserving a fiber bundle tower of $Y$ and projecting to a periodic map on its base. Then there exists a translational perturbation $f$ of $f_0$ such that $\Z^\infty(f)$ contains a subgroup isomorphic to either $C^\infty_c(I)$ or $\Diff^\infty_c(I)$. 
\end{coro}

\subsection{Reduction to the case when $H$ is trivial}
Let $X=G/B$ be a compact homogeneous space, and let $f_0$ be an affine map of $X$ that is not a $K$-system. We denote by $H$ the hyperbolically generated subgroup of $f_0$. Then it follows from \cite{PS04} that $\overline{H\cdot B}\neq G$. To prove Theorem \ref{thm: not K imp big cent}, the following discussion shows that without loss of generality we may assume $H$ is trivial.

Consider the following (nontrivial) fiber bundle of compact homogeneous spaces 
$$(\overline{H\cdot B})/B\to G/B\to G/(\overline{H\cdot B})\cong \frac{G/H}{(\overline{H\cdot B})/H},$$ and write $G_H=G/H$ and $B_H=(\overline{H\cdot B})/H$. It is clear that this fiber bundle structure is $f_0$-invariant. Moreover, the affine map $f_0$ induces an action $f_{0,H}$ on the base space $X_H=G_H/B_H$ whose hyperbolically generated subgroup is trivial. 

\begin{remark}\label{connectedHBfiber}
It should be noted that by Lemma \ref{connectedcomp} in Appendix~\ref{app: A}, the fiber $$(\overline{H\cdot B})/B=(\overline{H\cdot B})^0\cdot B/B\cong (\overline{H\cdot B})^0/(\overline{H\cdot B})^0\cap B$$ is connected.
\end{remark}

\begin{coro}\label{coro: H trivial coro}
    Assume that there exists a translational perturbation $f_H$ of $f_{0,H}$ such that $f_H$ preserves a fiber bundle tower of $Y$ and projects to a periodic map on its base. Then there exists an arbitrarily small perturbation $f$ of $f_{0}$ by left translation such that $\Z^\infty(f)$ contains a subgroup isomorphic to either $C^\infty_c(I)$ or $\Diff^\infty_c(I^k)$. 
\end{coro}
\begin{proof}
    This follows from Corollary \ref{coro: combine strategy}.
\end{proof}

Having described our strategy and preliminary techniques, we now move to the proof of 
 Theorem~\ref{thm: liecent iff ergodic} for specific classes of Lie groups.  In this section, we will cover the cases where $G$ is solvable, compact semisimple, and  semisimple without compact factors.  The remaining cases will be treated in the following section.
 
We start with the case where $G$ is solvable.  The proof naturally breaks into three steps of increasing generality: the cases where $G$ is abelian, nilpotent, and general solvable.

\subsection{The case when $G$ is abelian}\label{sec: abelian case}

We start with the torus case, generalizing the constructions in Examples~\ref{ex=torus1} and \ref{ex=torus2} and illustrating the basic strategy of our proof (i.e. constructing a suitable tower allowing for a translational perturbation of the base).  

We begin with the following lemma from linear algebra:

\begin{lemma}\label{lem: key alg num lem }
  For any $A\in\mathrm{GL}(n,\ZZ)$
  and any ordering $\{g_i(X)\}_{i=1}^r$ 
  for the $\ZZ$-irreducible factors of its characteristic polynomial, 
  $A$ is similar in $\mathrm{GL}(n,\ZZ)$ to a block triangular matrix 
  $\begin{pmatrix}
      A_{11} & &\\
     * &\ddots &\\
     * & * &A_{rr}
  \end{pmatrix}$,
  where the characteristic polynomial of each diagonal block $A_{ii}$ is $g_i(X)$.
\end{lemma}
\begin{proof}
See \cite{New72}, Theorem III.12.
\end{proof}



\begin{proof}[Proof of Theorem \ref{thm: not K imp big cent} for the torus case]\label{proof: torus case}
We write $f_0:\TT^n\to \TT^n$ as $f_0(x)=a+Ax,\;x\in \TT^n$ for some $a\in\RR^n$ and $A\in \mathrm{GL}(n,\ZZ)$, where $f_0$ is not a $K$-system. Since $f_0$ is not stably ergodic, the matrix $A$ has a root of unity as an eigenvalue. There are two cases.

{\bf Case 1:} $1\in\sigma(A)$.~\\
By induction, we obtain that $A$ is similar in $\mathrm{GL}(n,\ZZ)$ to a block triangular matrix $\begin{pmatrix}
      I_{d_1\times d_1} &0 \\
      (A_{21})_{d_2\times d_1}  &(A_{22})_{d_2\times d_2}
  \end{pmatrix}$.
Choose $a'$ close to $a$ in $\RR^n$ whose first $d_1$ components are in $\QQ$, and consider the fiber bundle
\[ \TT^{d_2}\to\TT^{n}\to \TT^{d_1} \]
Then the  translational perturbation $f=L_{a'}\circ A = L_{a'-a}\circ f_0$ projects to a periodic map $\bar{f}=L_{a'_1}\circ I_{d_1\times d_1}$ on the base torus $\TT^{d_1}$. 



{\bf Case 2:} $1\notin\sigma(A)$.~\\
By Lemma \ref{eigen1fix} implies that $f_0$ has a fixed point, and so $f_0$ is conjugate to its linear part by a translation. Hence we may assume that $a=0$. Lemma \ref{lem: key alg num lem }  implies that that $A$ is similar via $\mathrm{GL}(n,\ZZ)$ to a block triangular matrix
  $$\begin{pmatrix}
      (A_{11})_{d_1\times d_1} & &\\
      *&\ddots &\\
     * & * &(A_{rr})_{d_r\times d_r}
  \end{pmatrix},$$
where the characteristic polynomial of each diagnoal block $A_{ii}$ is irreducible over $\ZZ$, and $A_{11}$ has a root of unity as an eigenvalue. This forces the characteristic polynomial of $A_{11}$ to be a cyclotomic polynomial, 
  which means that some power of $A_{11}$ equals $I_{d_1\times d_1}$. Now consider the fibration
\[ \TT^{d_2}\to\TT^{n}\to \TT^{d_1}. \]
Then $f=A$ projects to a periodic map $\bar{f}=A_1$ on the base torus $\TT^{d_1}$. 

In both cases, there is a translational perturbation $f$ of $f_0$ preserving a toral fiber bundle and projecting to a periodic map $\bar{f}$ on the base torus. The conclusion of Theorem \ref{thm: not K imp big cent} follows from Corollary \ref{coro: 19+13+15 package}.
\end{proof}

Combining the torus case with Corollary \ref{coro: combine strategy}, we also obtain:

\begin{coro}\label{coro: reduction to torus base}
    Let $X=G/B$ be a compact homogeneous space, and let $f_0$ be an affine map of $X$ whose hyperbolically generated subgroup is trivial. Assume that $f_0$ preserves a fiber bundle tower whose base is a torus. Then there exists a translational perturbation $f$ of $f_{0}$ such that $\Z^\infty(f)$ contains a subgroup isomorphic to either $C^\infty_c(I)$ or $\Diff^\infty_c(I)$.
\end{coro}

\subsection{The case when $G$ is nilpotent}\label{nilpotent case}

\begin{proof}[Proof of Theorem \ref{thm: not K imp big cent} for $G$ nilpotent]
Let $X=G/B$ be a compact homogeneous space, where $G$ is a connected nilpotent Lie group, and $B\subseteq G$ is a closed cocompact subgroup. Suppose that $f_0\in \Aff(G/B)$ is not a $K$-system. To prove Theorem \ref{thm: not K imp big cent},  we may assume that the  hyperbolically generated subgroup  for $f_0$ is trivial, by Corollary \ref{coro: H trivial coro}. Then in view of Remark \ref{reducetononnormal} in Appendix~\ref{app: A}, we may also assume that $B$ doesn't contain a nontrivial closed subgroup that is normal in $G$. 

\begin{lemma}\label{uppercent}
    Let $G$ be a connected nilpotent Lie group, and let $B\subseteq G$ be a closed cocompact subgroup. Write $$
    \{1\}=Z_0(G)\subsetneqq Z_1(G)\subsetneqq\cdots \subsetneqq Z_u(G)=G 
    $$ for the upper central series of $G$. Then for each $0\leq i\leq u$, $Z_i(G)$ is a characteristic subgroup of $G$, and $B\cap Z_i(G)$ is a closed cocompact subgroup of $Z_i(G)$.
\end{lemma}
\begin{proof}
    See \cite[Chapter II]{Rag72}.
\end{proof}
Let
$\ell :=\max\{0\leq i< u: G\neq Z_i(G)\cdot B\}$, 
so that $G=Z_{\ell +1}(G)\cdot B$, and  $\overline{X}=G/(Z_\ell (G)\cdot B)$ is isomorphic to a torus. Lemma \ref{uppercent} then gives the following $f_0$-invariant fiber bundle tower of compact homogeneous spaces:
$$
\begin{tikzcd}
  (Z_1(G)\cdot B)/B \arrow[r, ] &  X=G/B \arrow[d,] \\
   (Z_2(G)\cdot B)/(Z_1(G)\cdot B) \arrow[r, ] & G/(Z_1(G)\cdot B) \arrow[d, ] \\
  \dots \arrow[r, ] & \dots\arrow[d, ] \\
  (Z_\ell(G)\cdot B)/(Z_{\ell-1}(G)\cdot B) \arrow[r, ] &G/(Z_{\ell-1}(G)\cdot B)\arrow[d,]\\
  & \overline{X}=G/(Z_\ell(G)\cdot B)
\end{tikzcd}
$$
The conclusion of Theorem \ref{thm: not K imp big cent} then follows from Corollary \ref{coro: reduction to torus base}.
\end{proof}

This completes the proof of Theorem~\ref{thm: not K imp big cent} in the case where $G$ is nilpotent.

\subsection{The case when $G$ is solvable}

\begin{proof}[Proof of Theorem \ref{thm: not K imp big cent} for the solvable case]\label{solvable case}
Let $X=G/B$, where $G$ is a connected solvable Lie group, and $B\subseteq G$ is a closed cocompact subgroup. Suppose that $f_0\in\Aff(X)$ is not a $K$-system.   To prove Theorem \ref{thm: not K imp big cent} for $f_0$, we may assume that its hyperbolically generated subgroup is trivial by Corollary \ref{coro: H trivial coro}. Again by Remark \ref{reducetononnormal} in Appendix~\ref{app: A}, we may assume that $B$ doesn't contain a nontrivial closed subgroup that is normal in $G$. 

By the Mostow structure theorem (see Lemma \ref{Mostowstructure} in Appendix~\ref{ss=mostow}), we have the fiber bundle sequence 
\begin{equation}\label{Mosfiber}
N/(N\cap B)\to G/B \to G/(N\cdot B)
\cong\dfrac{G/N}{(N\cdot B)/N}, 
\end{equation}
where $N$ is the nilradical of $G$. It is clear that this fiber bundle structure is preserved by the affine automorphism $f_{0}$.

If $G=N\cdot B$, then $G/B\cong N/(N\cap B)$, and we are reduced to the case where the Lie group is nilpotent. If $G\neq N\cdot B$, then the base space $\overline{X}=G/(N\cdot B)$ is a torus, and the conclusion of Theorem \ref{thm: not K imp big cent} follows from Corollary \ref{coro: reduction to torus base}.
\end{proof}

\subsection{The case when $G$ is semisimple without compact factors}

\begin{proof}[Proof of Theorem \ref{thm: not K imp big cent} for $G$ semisimple without compact factors]\label{noncompact semisimple case} Let $X=G/B$, where $G$ is a connected semisimple Lie group without compact factors, and $B\subseteq G$ is a  closed subgroup such that $G/B$ admits a $G$-invariant probability measure. Suppose $f_0\in\Aff(X)$ is not a $K$-system.  We we may assume that the hyperbolically generated subgroup of $f_0$ is trivial, by Corollary \ref{coro: H trivial coro}.
We will use the following lemma, which we prove in Section~\ref{app: A.4} of Appendix~\ref{app: A}.

\begin{lemma}[Finite liftability lemma]\label{finlift}
   Suppose that there is a finite cover of homogeneous spaces  $$  G'/\Gamma'\to G/\Gamma\to G''/\Gamma''   $$
   that is preserved by  $f_0 \in \Aff(G/\Gamma)$.  Let $\bar{f_0}$ denote the projection of $f_0$ to $G''/\Gamma''$.  If $\bar{f_0}$ can be approximated by a periodic affine automorphism with a constant linear part, then so can $f_0$.\end{lemma}

The next lemma reduces the proof to the case when $G$ is a product of connected noncompact linear simple Lie groups with trivial center, and $B = \Gamma$ is a lattice in $G$.

\begin{lemma}\label{linfin}
    Let $G$ be a connected semisimple Lie group without compact factors and let $B\subseteq G$ be a closed subgroup such that $G/B$ admits a $G$-invariant probability measure. Then the homogeneous space 
    $G/B$ is isomorphic to a finite cover of $\left(\prod\limits_{i=1}^rG_i\right)/\Gamma'$, where each $G_i$ is a connected noncompact linear simple Lie group with trivial center, and 
    $\Gamma'\subseteq \prod\limits_{i=1}^rG_i$ is a lattice.
\end{lemma}
\begin{proof} See Section~\ref{app: A.4} of Appendix~\ref{app: A}.
    \end{proof}

The next lemma reduces the proof to the case when the lattice is self-normalizing.

\begin{lemma}[\hbox{See \cite[Corollary 5.17]{Rag72}}]\label{normalizerandcenterlattice}
    Let $G$ be a connected semisimple Lie group without compact factor,
and $\Gamma\subseteq G$ be a lattice. Then the normalizer of $\Gamma$ in $G$ is also a lattice. In particular, if $Z(G)$ is the center of $G$, then
$Z(G)\cdot\Gamma$ is also a lattice in $G$.
\end{lemma}

Now we come to the perturbation lemma we will need in this context.

\begin{lemma}\label{semifinper}
    Let $G$ be a connected semisimple Lie group with finite center, and let $\Gamma\subseteq G$ be a lattice such that $N_G(\Gamma)\subseteq G$ is also a lattice. Then any affine automorphism of $G/\Gamma$ with trivial hyperbolically generated subgroup has a translational perturbation to a periodic affine automorphism.
\end{lemma}
Consider the fiber bundle 
$$
N_G(\Gamma)/\Gamma\to G/\Gamma\to G/N_G(\Gamma),
$$
where by Lemma \ref{normalizerandcenterlattice}, $N_G(\Gamma)/\Gamma$ is a finite group.
By Lemma \ref{finlift}, 
we may assume that $N_G(\Gamma)=\Gamma$.
The semisimplicity of $G$ implies that
 $|\mathrm{Aut}(G)/\mathrm{Inn}(G)|=:n(G)<+\infty$.

Now for any affine automorphism $f_0=L_aA$ of $G/\Gamma$,
and any positive integer $k$,
we have $f_0^k=a_k\cdot A^k$ 
where $a_k:=aA(a)\cdots A^{k-1}(a)\in G$.
When $k\in n(G)\cdot \mathbb{Z}$,
it follows that $f_0^k=L_{a_k}\circ c_b=L_{a_kb}\circ R_{b^{-1}}$ for some $b\in G$, where $c_b\in \Aff(G/\Gamma)$ denotes the conjugation by $b$.
Since $A^k=c_b$ preserves $\Gamma$,
we have $b\in N_G(\Gamma)=\Gamma$, and hence
$f_0^k$ is actually a  left translation: $f_0^k = L_{a_kb}$.

By  \cite[Theorem 1.1]{BS97},
the element $a_{k}b$ can be arbitrarily closely approximated by $x_k\in G$ such that $\mathrm{Ad}(x_k)$ has finite order. 
By the assumption that $G$ has finite center, the element $x_k\in G$ already has finite order and hence $L_{x_k}$ is periodic.

Note that the map $a\mapsto a_{k}b$ is open.
It follows that the above approximation to $a_{k}b$ can be realized as an approximation to $a$. 
We conclude that the affine automorphism $f_0$ can be perturbed to a periodic affine automorphism.
\end{proof}

In our setting, by Lemma \ref{finlift} and Lemma \ref{linfin}, we may assume that $B$ is a lattice. By Lemma \ref{normalizerandcenterlattice} and \ref{semifinper}, $f_0$ can be (translationally) perturbed to a periodic affine automorphism. 
Then the conclusion of Theorem \ref{thm: not K imp big cent} follows from Lemma \ref{wildperiod}.


\subsection{The case when $G$ is compact semisimple}
\begin{proof}[Proof of Theorem \ref{thm: not K imp big cent} for the compact semisimple case]\label{compact semisimple case} Let  $X=G/B$, where $G$ is a connected compact semisimple Lie group, and $B\subseteq G$ is a closed subgroup.  It holds automatically that the hyperbolically generated subgroup for {\em any} affine automorphism $f_0$ of $X$ is trivial. In  this case there are several ways to prove Theorem \ref{thm: not K imp big cent}. We choose the one that will be used later in a more general situation.

It is easy to see that there exist $k\in \ZZ^+$ and $x_k(a)\in G$ such that $$f_0^k=L_{x_k(a)b}\circ R_{b^{-1}},$$
for some $b\in N_G(B)$.
Since the map $a\mapsto x_k(a)$ is open, 
there is a perturbation $a'$ of $a$ such that $x_k(a')\cdot b$ is a torsion element in $G$ with a large prime $p=p(a')$ as a period. Let $f:=L_{a'}\circ A$.

Now consider  the minimal closed group generated by $b$, which we denote by $T_b$. Since $b\in N_G(B)$, the right translation by the elements in $T_b$ in $X=G/B$ is well-defined.  The subsets $\{xT_bB\subset X\}_{x\in G}$ (the set of orbits of $xB\in X$ by the right action of $T_b$) form a partition of $X$ that is invariant under left translation by elements of $G$. We claim that this partition is actually $f$-invariant.
\begin{proof}[Proof of the claim] Observe that $f^k=L_{(x_k(a')b)}\circ R_{b^{-1}}$ commutes with $R_{b^{-1}}$ on $X$. Therefore the orbit of $xB\in X$ under the dynamics of $f^{pk}$ is just $xT_{b^{p}}B$, where $T_{b^p}$ is the closed subgroup generated by $b^p$ (since $(x_k(a')b)$ has order $p$).
In particular, the partition $\{xT_{b^{p}}B\}_{x\in G}$ is $f$-invariant (since $f$ commutes with $f^{pk}$, $f$ permutes the orbit closure of $f^{pk}$).
But if $p$ is a large enough prime, then $T_b=T_{b^p}$, and therefore the partition $\{xT_{b}B\subset X\}_{x\in G}$ is also $f$-invariant.
\end{proof}
As a consequence of the claim, if we denote by $T_b^0$ the identity component of $T_b$, then for each $x\in G$, the connected component containing $xB$ of $xT_b\cdot B$ is a connected, closed submanifold of $X$ and therefore equals $xT_b^0\cdot B\subset X$ (because they have the same dimension and both are connected). Thus the partition $\{x T_b^0\cdot B\subset X\}_{x\in G}$ is actually a compact $f$-invariant foliation. In addition, we have $T_b^0\subset T_b\subset N_C(B)$. Consequently we obtain an $f$-invariant fiber bundle structure of compact homogeneous spaces $$T_b^0/(T_b^0\cap B)\to G/B \to G/(T_b^0\cdot B),$$ 
where the base map induced by $f$ is periodic. Then the conclusion of Theorem \ref{thm: not K imp big cent} follows from Corollary \ref{coro: 19+13+15 package}.
\end{proof}

We conclude that the conclusions of Theorem \ref{thm: not K imp big cent} hold when  $G$ is either solvable, compact semisimple, or semisimple without compact factors.

\subsection{Reduction to towers with specified bases} Using the results in this section, we extract a corollary that we will use in the following section to prove the general case of Theorem~\ref{thm: not K imp big cent}.

\begin{coro}\label{coro: tower reduction}
    Let $X=G/B$ be a compact homogeneous space, and let $f_0$ be an affine map of $X$ whose hyperbolically generated subgroup is trivial. Assume that $f_0$ preserves a fiber bundle tower with base $\bar X = \bar G/\bar B$, where $\bar G$ is either solvable, compact semisimple, or semisimple without compact factors.
    
   Then there exists a translational perturbation $f$ of $f_{0}$ such that $\Z^\infty(f)$ contains a subgroup isomorphic to either $C^\infty_c(I)$ or $\Diff^\infty_c(I)$.
\end{coro}

\begin{proof} This follows from the preceding proof of Theorem~\ref{thm: not K imp big cent} in the solvable and two semisimple cases and Corollary \ref{coro: combine strategy}.
\end{proof}

\section{The proof of Theorem \ref{thm: not K imp big cent} for general $G$ (mixed solvable and semisimple).}\label{sec: general case}

In this section we prove Theorem \ref{thm: not K imp big cent} in the most general case. Let $G$ be a connected Lie group and $B$ be a closed cocompact subgroup of $G$ with a finite invariant covolume. Let $f_0=L_a\circ A$ be an affine automorphism of the homogeneous space $X=G/B$ that is not a $K$-system. To prove Theorem \ref{thm: not K imp big cent}, we may assume that the hyperbolically generated subgroup for $f_0$ is trivial by Corollary \ref{coro: H trivial coro}.

For technical reasons, we introduce the notion  of ``admissible subgroup'' in our discussions. Historically, this terminology was used by Dani \cite{Dani77} to describe the closure of the image of a lattice in a Lie group under a continuous homomorphism. From this point of view, he derived many important properties of admissible subgroups, which are shared by lattices. One of them is of particular interest and wide use, so it becomes the modern definition \cite{BM81} of an admissible subgroup as follows. 

\begin{definition}[admissible subgroup]\label{def: admissible}
    Let $G$ be a connected Lie group. A closed subgroup $B\subseteq G$ is called admissible, if 
\begin{enumerate}
    \item $G/B$ admits a finite invariant volume.
    \item There exists a closed connected solvable subgroup $T$ of $G$,
    which contains the radical of $G$, is normalized by $B$ and such that $T\cdot B$ is closed. \label{item: T}
\end{enumerate}
\end{definition}

We will use the following criterion for a closed subgroup to be admissible, due to Morris. 
\begin{lemma}[$\hbox{\cite[Theorem 4.12]{Wit87}}$]\label{Admissible iff finite volume} 
Every closed subgroup with a finite invariant covolume in a Lie group is admissible.
\end{lemma}
 In particular any lattice is admissible. 

One can also extract from the proof of  \cite[Theorem 4.12]{Wit87} the following  lemma.

\begin{lemma}
    If $B\subseteq G$ is admissible, then the solvable subgroup $T$ in item \eqref{item: T} can be chosen so that for all  $A\in\mathrm{Aut}(G)$, if  $B$ is $A$-characteristic then so is  $T$.  (See Appendix~\ref{app: A} for a definition of a characteristic subgroup of a Lie group).
\end{lemma}

\begin{lemma}\label{imageadmissible}
    Let $\pi\colon G_1\to G_2$ be a continuous surjective homomorphism between Lie groups.  
\begin{enumerate}
    \item If $B_2\subseteq G_2$ is an admissible subgroup of $G_2$,
    then $B_1:=\pi^{-1}(B_2)$ is an admissible subgroup of $G_1$.
    \item If $B_1\subseteq G_1$ is an admissible subgroup of $G_1$,
    then $B_2:=\overline{\pi(B_1)}$ is an admissible subgroup of $G_2$.
\end{enumerate}
\end{lemma}
\begin{proof}
See \cite{BM81}, Proposition 2.1.
\end{proof}



\bigskip

Our proof will be divided into three cases.

\subsection{The case when $G$ is solvable}\label{sec: 5.1}
The conclusion of Theorem \ref{thm: not K imp big cent} follows from the proof of the solvable case in the previous section.

\subsection{The case when the semisimple part of $G$ is nontrivial and compact}\label{sec 5.2}
For the connected Lie group $G$, we write $R$ for its radical, $N$ for its nilradical, and $S$ for a Levi subgroup of $G$. Here $S$ is assumed to be nontrivial and compact. Inside $S$ we denote by $C_1$  the product of the factors of $S$ that are not normal in $G$, and $C_2$ for the product of the remaining, normal factors of $S$.

In view of Lemma \ref{unicover}, we may assume that $G$ is simply connected. Then $G$ is the semi-direct product of $R$ and $S$, and $S = C_1 \times C_2$. Moreover, $G$ is also the direct product of $C_1R$ and $C_2$.

Consider the fiber bundle of compact homogeneous spaces 
$$
C_2/(C_2\cap B)\to G/B\to G/(C_2\cdot B)\cong\dfrac{G/C_2}{(C_2\cdot B)/C_2},
$$
which is preserved by the affine automorphism $f_0$. Here $G/C_2$ has no nontrivial connected compact semisimple normal subgroup, and $(C_2\cdot B)/C_2$ is still a cocompact admissible subgroup of $G/C_2$, by Lemma \ref{imageadmissible}.
\\

\noindent{\bf Subcase 1}: $G=C_2\cdot B$. In this subcase we have 
$G/B\cong C_2/(C_2\cap B)$. Then the conclusion of Theorem \ref{thm: not K imp big cent} follows from Corollary~\ref{coro: tower reduction} (in the case where the base is compact semisimple).

\bigskip

\noindent{\bf Subcase 2:} $G\neq C_2\cdot B$. Write $G_1=G/C_2$ and $B_1=(C_2\cdot B)/C_2$. In the above direct product structure of $G$, we may identify $G_1$ with $C_1R$, and $B_1$ with a cocompact admissible subgroup of $C_1R$. In  view of Remark \ref{reducetononnormal}, we may also assume that $B_1$ contains no nontrivial closed, connected normal subgroup of $G_1$. Then by Lemma \ref{nilradicalhered}, $f_0$ preserves the following fiber bundle  of compact homogeneous spaces: 
$$
(C_1N)/(C_1N\cap B_1)\to G_1/B_1\to G_1/(C_1N\cdot B_1)\cong\dfrac{G_1/(C_1N)}{(C_1N\cdot B_1)/(C_1N)}.
$$
Here $G_1/(C_1N)\cong R/N$ is an abelian group, and $(C_1N\cdot B_1)/(C_1N)$ is a still a cocompact admissible subgroup of $G_1/(C_1N)$, by Lemma \ref{imageadmissible}.

\medskip

{\bf Subcase 2.1:} $G_1\neq C_1N\cdot B_1$. Write $G_2=G_1/(C_1N)$ and $B_2=(C_1N\cdot B_1)/(C_1N)$. Then $G_2/B_2$ is a torus, and $f_0$ preserves the following fiber bundle tower of compact homogeneous spaces:
$$\begin{tikzcd}
  C_2/(C_2\cap B)  \arrow[r,] & G/B \arrow[d, 
  ] \\
  (C_1N)/(C_1N\cap B_1)  \arrow[r,] & G_1/B_1 \arrow[d, 
  ] \\
 & G_2/B_2. 
\end{tikzcd}
$$
The conclusion of Theorem \ref{thm: not K imp big cent} follows from Corollary~\ref{coro: tower reduction} (in the case where the base is a torus).

\medskip

{\bf Subcase 2.2}: $G_1=C_1N\cdot B_1$. Then we have $G_1/B_1\cong (C_1N)/(C_1N\cap B_1)$.  Lemma \ref{nilradicalhered} implies that the following fiber bundle structure of compact homogeneous spaces is preserved by $f_0$:
$$
N/(N\cap B_1)\to G_1/B_1\to G_1/(N\cdot B_1)\cong C_1/(C_1\cap NB_1).
$$

\smallskip

\hspace{7pt}{\bf Subcase 2.2.1}: $G_1\neq N\cdot B_1$. Then $f_0$ preserves the tower
$$\begin{tikzcd}
  C_2/(C_2\cap B)  \arrow[r,] & G/B \arrow[d, 
  ] \\
  N/(N\cap B_1)  \arrow[r,] & G_1/B_1 \arrow[d, 
  ] \\
 & G_2/B_2 
\end{tikzcd}
$$
The conclusion of Theorem \ref{thm: not K imp big cent} follows from Corollary~\ref{coro: tower reduction} (in the case of torus base).

\smallskip

\hspace{7pt}{\bf Subcase 2.2.2:} $G_1=N\cdot B_1$. In this subcase have $G_1/B_1\cong N/(N\cap B_1)$. Then $f_0$ preserves a fiber bundle tower of compact homogeneous spaces
$$
C_2/(C_2\cap B_1)\to G/B\to G_1/B_1,
$$
and the conclusion of Theorem \ref{thm: not K imp big cent} follows from Corollary~\ref{coro: tower reduction} (where the base is a nilmanifold). 

This completes the proof of Subcase 2.2, and hence of Case 2.  We have proved Theorem~\ref{thm: liecent iff ergodic} in the case where the semisimple part of $G$ is nontrivial and compact.

\subsection{The case when the semisimple part of $G$ is noncompact}  
In this subsection, we consider the remaining case where the Levi subgroup $S$ is noncompact.  Note that this includes the semisimple case not covered in the previous section (i.e., the case where $G$ is sempsimple and contains both compact and noncompact factors).

Note that  $B$ is a cocompact admissible subgroup of $G$ by Lemma \ref{imageadmissible}. Write $R$ for the radical of $G$
and $T$ for the closed, connected, solvable subgroup of $G$ such that 
\begin{enumerate}
    \item $T$ contains $R$;
    \item $T$ is normalized by $B$, and $T\cdot B$ is closed in $G$,
\end{enumerate}
which exists by the admissability of $B$.
It is clear from the definition that $T\cdot B$ is also a cocompact admissible subgroup of $G$.
Moreover, $T\cap B$ is a closed cocompact subgroup of $T$ with a finite invariant covolume, by Lemmas~\ref{heredecomp} and \ref{propertyPfiber} in Appendix~\ref{app: A}.
Consider the $f_0$-invariant fiber bundle structure of compact homogeneous spaces
$$
T/(T\cap B)\to G/B\to G/(T\cdot B)\cong \dfrac{G/R}{(T\cdot B)/R},
$$
If $G=T\cdot B$, then $G/B\cong T/(T\cap B)$  is a solvmanifold, and the theorem follows from 
\ref{sec: 5.1}. Thus we may assume that $G\neq T\cdot B$. Then $G/R$ is a connected, semisimple Lie group, and $(T\cdot B)/R$ is a cocompact admissible subgroup of $G/R$, by Lemma~\ref{imageadmissible}.

Let $G_1=G/R$, $B_1=(T\cdot B)/R$,
and let  $C$  be the compact semisimple factor of $G_1$.
Since $C\cdot B_1$ is closed in $C$ and $C\cap B_1$ is cocompact in $C$, 
we have the following $f_0$-invariant fiber bundle structure of compact homogeneous spaces
$$
C/(C\cap B_1)\to G_1/B_1\to G_1/(C\cdot B_1)\cong 
\dfrac{G_1/C}{(C\cdot B_1)/C}.
$$

If $G_1=C\cdot B_1$, then  $G=C'T\cdot B$, where $C'$ is a compact semisimple factor of $G$. It follows from Lemma \ref{propertyPfiber} that $G/B\cong C'T/(C'T\cap B)$, and we reduce to the case where the semisimple part is noncompact.  The result then follows from Section \ref{sec 5.2}.

Thus we may assume that $G_1\neq C\cdot B_1$. Here $G_1/C$ is a connected semisimple Lie group without compact factors, and 
$(C\cdot B_1)/C$ is a cocompact admissible subgroup of $G_1/C$,
by Lemma \ref{imageadmissible}.

Writing $G_2=G_1/C$ and $C_2=(C\cdot B_1)/C$,
we obtain the $f_0$-invariant tower
$$\begin{tikzcd}
  T/(T\cap B)  \arrow[r,] & G/B \arrow[d, 
  ] \\
C/(C\cap B_1)  \arrow[r,] & G_1/B_1\arrow[d, 
  ] \\
& G_2/B_2
\end{tikzcd}
$$
The conclusion of Theorem \ref{thm: not K imp big cent} then follows from Corollary~\ref{coro: tower reduction} (in the case where the base is semisimple, without compact factors).


%
%


%


\appendix

\section{Some facts from Lie theory}\label{app: A}
\subsection{Characteristic subgroups}
We first recall the definition of a characteristic subgroup of a connected Lie group.
\begin{definition}[Characteristic subgroup]
    Let $G$ be a connected Lie group, and let $B\subseteq G$ be a closed, connected subgroup. For an automorphism $A$ of $G$, 
    $B$ is said to be {\em $A$-characteristic} if $A(B)=B$.
    We say that $B$ {\em characteristic} if $B$ is $A$-characteristic
    for all automorphisms $A$ of $G$.
\end{definition}
For a general connected Lie group $G$,
its radical $R$ is the unique maximal connected solvable normal subgroup,
and its nilradical $N$ is the unique maximal connected nilpotent normal subgroup.
In $G$ there are maximal connected semisimple subgroups and all of them are conjugate; if $S$ is one of them, then $G=R\cdot S$.
Moreover, if $G$ is simply connected, then $R\cap S=\{e\}$ and the product is semi-direct. This is called the {\em Levi-Malcev decomposition} of $G$.

For a given $G$, we write $Q$ for the product of all noncompact factors of $S$, and 
$C$ for the product of all compact factors of $S$.
In $C$, we define $C_1$ to be the product of the factors of $C$ that are not normal in $G$ and $C_2$ to be the product of the factors of $C$ that are normal in $G$. Then $C_2$ is characterized as the {\em unique } maximal connected, compact semisimple normal subgroup of $G$.

\begin{lemma}[$\hbox{See \cite[Lemma 1.1]{Iwa}}$]\label{charsubgroup}
 Using the  notation above, the subgroups  $R$,$N$,$C_2$,
 $C_1R$,$C_1N$ and $QC_1R$ are all characteristic in $G$.
\end{lemma}

\subsection{Reductions on homogeneous spaces}

In many cases, we expect that a closed cocompact subgroup $B$ of $G$ with a finite covolume to have similar properties to a lattice. In order for this to be true, we need to ``remove" any closed subgroups of $B$ that are normal in $G$ to obtain new groups $G', B'$ with the same quotient.  The following lemma makes this precise.
\begin{lemma}\label{normalinadmissible}
Let $G$ be a connected Lie group, and let $B\subseteq G$ be a closed subgroup. 
Then $B$ contains a unique maximal closed subgroup $B_0$
that is normal in $G$. Moreover, $B_0$ is $A$-characteristic for every automorphism $A$ of $G$ satisfying $A(B)=B$.

\end{lemma}
\begin{proof}
    Let $B_0\subseteq B$ be a maximal closed subgroup that is normal in $G$. 
If there were another such maximal subgroup $B_0'$, then  $\overline{B_0\cdot B_0'}$ would be yet another,
which implies that $B_0=\overline{B_0\cdot B_0'}=B_0'$ by maximality of $B_0$ and $B_0'$.
Thus $B_0$ is unique. 

Let $A$ be an automorphism preserving $B$.  Since $A(B_0)\subseteq B$ is also a maximal closed subgroup  normal in $G$, by uniqueness we conclude that $A(B_0)=B_0$.
\end{proof}

\begin{remark}\label{reducetononnormal}
    In view of Lemma \ref{normalinadmissible}, we have an isomorphism of homogeneous space
    $$
    G/B\cong \dfrac{G/B_0}{B/B_0},
    $$
    which is preserved by any automorphism $A$ of $G$ with $A(B)=B$.
    Here $B/B_0$ contains no nontrivial closed, normal subgroup of $G/B_0$.
Thus quotienting by the maximal closed subgroup of $B$ that is normal in $G$, we may always assume that $B$ contains no nontrivial closed normal subgroup of $G$.
\end{remark}

The following lemma, also concerning the reduction of homogeneous spaces, is used in the proofs of our key propositions (\ref{prop: key prop 1} and \ref{prop: key prop 2}).
\begin{lemma}\label{nontrivialvectorfield}
    Let $X=G/B$ where $G$ is a connected Lie group and $B$ is a cocompact closed subgroup in $G$. Suppose that there exists a right-invariant vector field $V$ on $G$ that is generated by an element $v\in\mathfrak{g}=\mathrm{Lie}(G)$. Then its projection 
    $\overline{V}$ onto $X$ is nontrivial if and only if 
    the smallest closed normal subgroup containing $\{\exp(tv)\}_{t\in\RR}$ in $G$ isn't contained in $B$.
\end{lemma}
\begin{proof}
    Through any point $gB\in X$, the flow generated by $\overline{V}$ is given by $t\mapsto \exp(tv)\cdot gB$.
    It follows that $\overline{V}$ is trivial on $X$ if and only if the following condition holds:
    $$
    \exp(tv)\cdot gB=gB,\;\forall t\in\RR,\;\forall g\in G.
    $$
    This is equivalent to the statement that
    the smallest closed normal subgroup containing $\{\exp(tv)\}_{t\in\RR}$ in $G$ is contained in $B$.
\end{proof}

\subsection{Proof of Lemma \ref{lem: cllpse proc imp spcl fb}: Constancy of linear parts }\label{app: Lemma 18}
To prove Lemma \ref{lem: cllpse proc imp spcl fb}, we first consider an elementary fact about periodic diffeomorphisms.
\begin{lemma}\label{wildperiod}
Let $M$ be a connected, closed smooth manifold, and let $f$ be a periodic diffeomorphism of $M$ with minimal period $d$. Then there exists an open dense set $\mathfrak U\subset M$ such that for any $x\in \mathfrak U$, there exists a nonempty open neighborhood $U\subset M$ containing $x$ with $U\cap f^j(U)=\emptyset$ for all $1\leq j\leq d-1$. In particular, any diffeomorphism $g$ supported on $U$ extends to a diffeomorphism (with support in $\bigcup_{j=0}^{d-1}f^j(U)$) that commutes with $f$.
\end{lemma}
\begin{proof}
Take $S\subset \ZZ^+$ to be the set formed by all $n\in \ZZ^+$ such that $n|d$, and let $Y:=\bigcap_{n\in S}\{x\in M: f^n(x)\neq x\}$. A classical result of Newman \cite{New31} states that the set of fixed points of any periodic homemorphism $g$ on a connected manifold has empty interior unless $g=id$. Therefore $Y$ is a finite intersection of open dense set (since $d$ is the minimal period) and hence is open dense as well. In particular, $Y$ is non-empty. Let $y\in Y$ be arbitrary,  Since $f^j(y)\neq y, 1\leq j\leq d-1$, there is a neighborhood $U$ of $y$  such that $f^j(U)\cap U=\emptyset, 1\leq j\leq d-1$. 

For any $g$ supported in $U$, we extend $g$ to a diffeomorphism supported in  $\cup_{j=0}^{d-1} f^j(U)$ via conjugation, setting $g|_{f^j(U)}$ to be $f^j \circ g|_U\circ f^{-j}|_{f^j(U)}$. Since $\supp g|_U\subset U$, the newly defined extension of $g$ is smooth on $M$ (since $f^i(U)\cap f^j(U)=\emptyset, 1\leq i<j\leq d-1$), and by construction,  $g$ commutes with $f$.\end{proof}
\begin{proof}[Proof of Lemma \ref{lem: cllpse proc imp spcl fb}]We only need to prove the following lemma. \begin{lemma}\label{locallyconstantlinearpart}
  Let $G$ be a connected Lie group, let
  $B\subseteq G$ be a proper closed cocompact subgroup,
  and  let $H\subseteq G$ be a closed subgroup such that $H\cdot B=B\cdot H$.
  Suppose that $f$ is an affine automorphism of $G$ with linear part $A$,
  such that both $B$ and $H$ are $A$-characteristic,
  and that $\bar{f}$ is $d$-periodic on $G/(\overline{H\cdot B})$.
  Then there exists an open dense subset of $G/(\overline{H\cdot B})$,
  in which each point $x$ admits a sequence of  small open neighborhoods ${U_x},\bar{f}({U_x}),\cdots,\bar{f}^{d-1}({U_x})$ along its orbit, over which we may choose a sequence of local trivializations such that the linear parts of $f$ along fibers over ${U_x},\bar{f}({U_x}),\cdots,\bar{f}^{d-1}({U_x})$  are constant on each component.
\end{lemma}
\begin{proof}
First we assume that $B=\{e\}$, and $H$ is a proper closed cocompact subgroup in $G$. Then we have a simple fiber bundle sequence
$$
H\to G\xrightarrow{\pi} G/H.
$$
For each $g\in G$, we write $\overline{g}:=\pi(g)\in G/H$.

Suppose $x = \pi(g_0)$ is given, and suppose there exists a neighborhood $V_0$ of $g_0$ such that 
\begin{enumerate}
    \item $\pi(V_0)\cap \bar f^j(\pi(V_0)) = \varnothing$, for $j=1,\ldots, d-1$;
    \item there exists a section $s_0\colon \pi(V_0) \to G$  of $\pi$ over $\pi(V_0)$ (so that $\pi\circ s_0=\mathrm{id}_{\pi(V_0)}$).
\end{enumerate}
Such a neighborhood exists for an open-dense set of $x\in G/H$.

Fix such an $x$ and $V_0$ with local section $s_0\colon \pi(V_0) \to G$. Let $U_x:= \pi(V_0)$.
Since $s_0\circ \pi$ is smooth in $V_0$,
we may write
$s_0(\overline{g})=g\cdot h_0(g)$ for each $g\in V_0$,
where $h_0\colon V_0\to H$ is smooth.
Consider the following ``local bundle" over $U_x$:
\[H\hookrightarrow {\mathcal B}_0:= \bigcup\limits_{g\in V_0}gH \stackrel{\pi}{\to} U_x\]
(note that $\pi^{-1}(U_x) \neq  V_0$ in general).
Using this section $s_0$, we define a  trivialization of this bundle by
$$
 \bigcup\limits_{g\in V_0}gH\to s_0(U_x)\times H,
\quad
gh\mapsto (s_0(\overline{g}),h_0(g)^{-1}h).
$$

Similarly,  consider the local bundle 
\[{\mathcal B}_1:= f\left (\bigcup\limits_{g\in V_0}gH \right) = \bigcup\limits_{g\in V_0}f(g)H.
\]
Write $f:= L_a\circ A$, $a\in G$ and $A\in\mathrm{Aut}(G)$,
and set
$$
s_1:=f\circ s_0\circ \bar{f}^{-1},
\quad h_1:=A\circ h_0\circ f^{-1}.
$$
Then $s_1$ is a section of ${\mathcal B}_1$, which defines a trivialization of ${\mathcal B}_1$ by
$$
\bigcup\limits_{g\in V_0}f(g)H\to s_1(\bar{f}(U_x))\times H,
\quad
f(g)A(h)\mapsto (s_1(\bar{f}(\overline{g})),h_1(f(g))^{-1}A(h)).
$$

These trivializations have the property that the restriction 
$f\colon {\mathcal B}_0\to {\mathcal B}_1$
becomes a product map
$$f\times A\colon s_0(U_x)\times H\to s_1(\bar{f}(U_x))\times H.$$

Repeating this argument iteratively, setting ${\mathcal B}_k:= f^k({\mathcal B_0})$, for $k=0,\ldots,d-1$, we obtain disjoint open neighborhoods $f^{k}(U_x)$ and local sections $s_k$ defined by
$$
s_k:=f^k\circ s_0\circ \bar{f}^{-k},
\quad h_k:=A^k\circ h_0\circ f^{-k}
$$
such that, in the trivializing coordinates
$$
\bigcup\limits_{g\in V_0}f^k(g)H\to s_k(\overline{f^k}(U_x))\times H,
\quad
f^k(g)A(h)\mapsto (s_k(\bar{f}^k(\overline{g})),h_k(f^k(g))^{-1}A^k(h)),
$$
the restriction $f\colon {\mathcal B}_k\to {\mathcal B}_{k+1}$
has constant linear part $A$, for $k=0,\ldots, d-2$.


Finally, we need to check the linear part $f\colon {\mathcal B}_{d-1}\to {\mathcal B}_0$.
Under the two coordinate changes defined by $(s_{d-1},h_{d-1})$ 
and $(s_0,h_0)$,
the affine automorphism $f$ becomes a skew-product map in local coordinates:
\begin{equation*}
    \begin{aligned}
        s_{d-1}\left(\bar{f}^{d-1}(U_x)\right)\times H
        &\to s_0(U_x)\times H
        \\
        (f^{d-1}(s_0(\overline{g})),
        A^{d-1}(h_0(g))^{-1}\cdot A^{d-1}(h))
        &\mapsto 
        (s_0(\overline{g}),h_0(f^d(g))^{-1}\cdot A^d(h)),
    \end{aligned}
\end{equation*}
where the action along the fiber has linear part $A$. 
This completes the proof in the case where $B$ is trivial.
\\

We now consider the general case where $B\subseteq G$ is a proper, closed cocompact subgroup, and that $H\subseteq G$ is a closed subgroup such that $H\cdot B$ is a subgroup.
Then we have a simple fiber bundle sequence
\begin{equation}\label{simpleHB}
    \overline{HB}\to G\xrightarrow{\pi} G/(\overline{HB}),
\end{equation}
and a further fiber bundle sequence
\begin{equation}\label{furtherHB}
\overline{HB}/B\to G/B\to G/(\overline{HB}).
\end{equation}
For the fiber bundle sequence (\ref{simpleHB}),
we apply the argument above to find for an open-dense set of $x\in G/(\overline{HB})$ a neighborhood $V_0$ of $g_0\in \pi^{-1}(x)$,
and a sequence of local coordinates in $\bigsqcup_{j=0}^{d-1}\bar f^j(U_x)$, where $U_x = \pi(V_0)$:  
$$
\bigcup\limits_{g\in V_0}g\cdot\overline{HB}\to s_0(U_x)\times \overline{HB},
\quad
g\xi\mapsto (s_0(\overline{g}),\xi_0(g)^{-1}\xi).$$
An important observation is that this coordinate change on $G$ descends to a coordinate change on $G/B$:
$$
\bigcup\limits_{g\in V_0}g\cdot\overline{HB}/B\to s_0(U_x)\times \overline{HB}/B,
\quad
g\xi\mapsto (s_0(\overline{g}),\xi_0(g)^{-1}\xi B).
$$
Then all properties of these series of local coordinates are preserved. In particular, in these local coordinates, the affine diffeomorphism along fibers has constant linear parts on connected components.  This completes the proof.
\end{proof}This completes the proof  of Lemma \ref{lem: cllpse proc imp spcl fb}\end{proof}
\subsection{Proof of Lemmas~\ref{finlift}  and \ref{linfin}}\label{app: A.4}
First we prove Lemma \ref{finlift}.
\begin{proof}[Proof of Lemma \ref{finlift}]
    Since $f_0$ preserves this covering structure, it restricts to an affine automorphism on the fiber modeled by $G'/\Gamma'$,
    which is a finite set. By raising to some power, we may assume that $f_0^k$ restricts to the identity on $G'/\Gamma'$.
    Since
    the map 
    $$
    G\to G'',\qquad a\mapsto \bar{a}^k
    $$
    is an open map, any perturbation of $\bar{f_0}^k$ on the translation part can be realized by a perturbation of $f_0$ on the translation part.
In particular, $f_0$ can be perturbed to a periodic affine automorphism with the same linear part.
\end{proof}
The proof of Lemma \ref{linfin} is slightly more complicated. First we reduce the problem to the case that $G$ is a product of connected linear simple Lie groups.
\begin{lemma}\label{unicover}
    Let $G$ be a connected Lie group and $B\subseteq G$ be a closed subgroup.  
    Let $\widetilde{G}$ denote the universal cover of $G$, and $Z\subseteq Z(\widetilde{G})$ the deck transformation group for this cover. Then $\widetilde{B}:=Z\cdot B\subseteq \widetilde{G}$ is also a closed subgroup, and the covering map gives an isomorphism between homogeneous spaces $\widetilde{G}/\widetilde{B}\cong G/B$. Moreover, if $G$ is connected semisimple, $B$ is cocompact, and $G/B$ has a $G$-invariant probability measure, then so do $\widetilde{G}$, $\widetilde{B}$, and $\widetilde{G}/\widetilde{B}$.
\end{lemma}
\begin{proof}
    This follows from the basic properties of a covering map.
\end{proof}
Second we reduce to the case that $B$ is a lattice. We recall the Zariski density property for a closed subgroup in a general connected Lie group (possibly  non-algebraic).

\begin{definition}[Zariski density]
Let $G$ be a connected Lie group. A closed subgroup $B\subseteq G$ is said to be {\em Zariski-dense in $G$} if for any finite-dimensional representation of $G$ in $\GL(n,\RR)$, the Zariski closure of $\rho(B)$ in $\M(n,\RR)$ equals that of $\rho(G)$. 
\end{definition}
The following lemma is just the classical Borel density theorem.
\begin{lemma}\label{Boreldensity}
    Let $G$ be a connected semisimple Lie group without compact factors,
and let  $B\subseteq G$ be a closed subgroup such that $G/B$ admits a $G$-invariant probability measure. Then $B$ is Zariski dense in $G$.
\end{lemma}
\begin{proof}
See \cite[Appendix]{Adams}. In fact, $B$ is Zariski dense in the closed subgroup generated by all split algebraic one-parameter subgroups of $G$.
\end{proof}
\begin{lemma}\label{Zardensenormal}
    Let $G$ be a connected Lie group, and $B\subseteq G$ be a closed, Zariski-dense subgroup. Then the identity component $B^0$ is a normal subgroup of $G$.
\end{lemma}
\begin{proof}
    It is clear that for the adjoint representation of $G$,
    the action of each $g\in B$ preserves the subalgebra $\mathfrak{b}=\mathrm{Lie}(B)$. By the Zariski density of $B$ in $G$,
    the action of each $g\in G$ also preserves the subalgebra $\mathfrak{b}$.
    This implies that the conjugation $c_g$ by each $g\in G$ preserves the identity component $B^0$, i.e. $B^0$ is a normal subgroup of $G$.
\end{proof}

\begin{remark}\label{reducetodiscrete}
    In view of Remark \ref{reducetononnormal} and Lemma \ref{Zardensenormal},
 we may assume that  every closed, Zariski-dense subgroup is discrete. 
\end{remark}

\begin{remark}\label{reducetolattice}
In view of Lemma \ref{Boreldensity} and Remark \ref{reducetodiscrete}, for a closed subgroup with a finite invariant covolume in a connected semisimple Lie group without compact factors, we may always assume that it is a lattice.
\end{remark}

The last step is to reduce to the case that $G$ has trivial center.
\begin{lemma}\label{trivialcenter}
    Let $G$ be a connected semisimple Lie group. Then its adjoint group $\mathrm{Ad}(G)$ has trivial center.
\end{lemma}
\begin{proof}
 It is known that the adjoint representation gives a group isomorphism $G/Z(G)\cong \mathrm{Ad}(G)$. Let $\overline{g}\in Z(\mathrm{Ad}(G))$ for $g\in G$. It follows that for any $h\in G$, we have $g^{-1}h^{-1}gh\in Z(G)$, so there is a continuous map 
 $$
 G\to Z(G),\qquad h\mapsto g^{-1}h^{-1}gh.
 $$
 Since $G$ is connected semisimple, its center $Z(G)$ is discrete. This forces the above map to be constant.
In particular, for any $h\in G$ we have $g^{-1}h^{-1}gh=1$,
i.e. $g$ lies in $Z(G)$.
\end{proof}
\begin{lemma}\label{simplelinear}
    Let $G$ be a connected semisimple Lie group without compact factors, and let $\Gamma\subseteq G$ be a lattice. Let $\mathrm{Ad}\colon G\to\mathrm{GL}(\mathfrak{g})$ denote the adjoint representation of $G$.
        Then $\mathrm{Ad}(G)$ is a connected semisimple Lie group
        without compact factors and with trivial center,
        $\mathrm{Ad}(\Gamma)\subseteq \mathrm{Ad}(G)$ is also a lattice, and the adjoint representation gives a finite cover
        $$
        (Z(G)\cdot \Gamma)/\Gamma \to G/\Gamma \to G/(Z(G)\cdot \Gamma)\cong \mathrm{Ad}(G)/\mathrm{Ad}(\Gamma).
        $$
\end{lemma}
\begin{proof}
    This follows from Lemma \ref{normalizerandcenterlattice} and Lemma \ref{trivialcenter}.
\end{proof}
Now we can prove Lemma \ref{linfin}.
\begin{proof}[Proof of Lemma \ref{linfin}] By Lemma \ref{unicover}, we may assume that $G$ is simply connected. Then by semisimplicity, $G$ can be written as a direct product of  connected simple Lie groups: $G\cong \prod\limits_{i=1}^{r}G_i$. In view of Remark \ref{reducetolattice},
    we may further assume that $B$ is a lattice in $G$.
    Then applying Lemma \ref{simplelinear} leads us to the conclusion.
\end{proof}


\subsection{Proof of Remark \ref{connectedHBfiber}}
We will need the following fact:
\begin{lemma}\label{connectedcomp}
    Let $G$ be a Lie group, and let $B\subseteq G$ be a closed subgroup.
Then $G/B$ is compact (resp. carries a $G$-invariant finite measure) 
if and only if $G^0/(G^0\cap B)$ is compact (resp. carries a $G^0$-invariant finite measure), where $G^0$ denotes the identity component of $G$.
\end{lemma}
\begin{proof}
    Since $B$ normalizes $G^0$, we see that $G^0\cdot B$ is an open and hence closed subgroup of $G$. If $G/B$ is either compact or carries a $G$-invariant finite measure, then the same is true of the discrete space $G/(G^0\cdot B)$, so that this set is finite.
    On the other hand, $(G^0\cdot B)/B$ is compact (resp. carries an invariant finite measure) if and only if $G^0/(G^0\cap B)$ is compact (resp. carries an invariant finite measure).
    The conclusion follows.
\end{proof}

\subsection{Hereditary properties of closed subgroups}

Our main strategy consists of decomposing a general Lie group into a characteristic subgroup and a quotient group by it. Unfortunately,  a general closed subgroup may not inherit this decomposition. We thus introduce the following notions.

\begin{definition}[Hereditary property]
Let $G$ be a connected Lie group, 
and let $H\subseteq G$ be a connected, closed subgroup of $G$.
For a closed subgroup $B\subseteq G$ with property $P$,
$H$ is called $(B,P)$-hereditary if the closed subgroup $B\cap H$ in $H$ has property $P$.
Furthermore, $H$ is called $P$-hereditary if $H$ is $(B,P)$-hereditary for all closed subgroups $B$ of $G$ with property $P$.
\end{definition}

\begin{remark}
    In our application ``property $P$'' will usually be cocompactness (in the solvable case), and having a finite invariant covolume (in the semisimple case), or both. 
    It is clear that a subgroup is $P\& P'$-hereditary if and only if 
    it is both $P$-hereditary and $P'$-hereditary.
    One could also consider ``property $P$'' to be ``is a lattice" or ``is an admissible subgroup."
\end{remark}

\begin{lemma}\label{heredecomp}
    Let ``property $P$'' be cocompactness,  having a finite invariant covolume, or both. Then for any two closed subgroups $H_2\subseteq H_1$ in $G$,
$H_2$ has property $P$ in $G$ if and only if $H_2$ has property $P$ in $H_1$ and $H_1$ has property $P$ in $G$.
\end{lemma}
\begin{proof}
See \cite{Rag72}, Lemma 1.6.
\end{proof}

\begin{lemma}\label{propertyPfiber}
    Let $G$ be a locally compact group, and let $H,B\subseteq G$ be  closed subgroups such that $H\cdot B=B\cdot H$.
Then 
    \begin{enumerate}
        \item if $H/(H\cap B)$ is compact, then $H\cdot B$ is closed; and
        \item if $H\cdot B$ is closed, then there is a natural homeomorphism 
    $H/(H\cap B)\to (H\cdot B)/B$;
   moreover if  ``property $P$'' is cocompactness, having a finite invariant covolume, or both, then $H\cap B$ has property $P$ in $H$ if and only if $B$ has property $P$ in $H\cdot B$.
    \end{enumerate}
\end{lemma}
\begin{proof}
See \cite{Rag72}, Lemma 1.7.
\end{proof}

\subsection{The Mostow structure theorem and  a corollary}
\label{ss=mostow}

\begin{lemma}\label{maximalnilpotent}
    Let $G$ be a connected Lie group whose semisimple part is compact and doesn't contain any nontrivial closed connected normal subgroup of $G$. Then its nilradical $N$ is the maximal connected nilpotent subgroup of $G$.
\end{lemma}
\begin{proof}
See \cite[Lemma 3.9]{Mos71} or \cite[Proposition 1.3]{Wu88}.
\end{proof}

\begin{lemma}[Mostow Structure Theorem]\label{Mostowstructure}
    Let $G$ be a connected solvable Lie group, and $B\subseteq G$ be a closed subgroup. Then
\begin{enumerate}
    \item $G/B$ carries a $G$-invariant finite measure if and only if $G/B$ is compact.
    \item Write $N$ for the nilradical of $G$. Suppose that
    $G/B$ is compact and that $N\cap B$ doesn't contain a nontrivial closed subgroup that is normal in $G$.
    Then $N/(N\cap B)$ is compact.
\end{enumerate}
\end{lemma}
\begin{proof}
See Chapter III of \cite{Rag72}.
\end{proof}

\begin{coro}
\label{nilradicalhered}
Let $G$ be a connected Lie group whose semisimple part is compact
and contains no nonontrivial closed connected normal subgroup of $G$, and let $B\subseteq G$ be a cocompact admissible subgroup.
 Suppose that $N\cap B$ contains no nontrivial connected normal subgroup of $G$, where $N$ is the nilradical of $G$. Then setting  $H$ to be either $N$ or $C_1N$, we have that 
$H\cap B$ is a cocompact admissible subgroup in $H$.
\end{coro}
\begin{proof}
Since $B\subseteq G$ is admissible, there exists a closed connected solvable subgroup $T$ of $G$,
    which contains $N$, is normalized by $B$ and such that $T\cdot B$ is closed. By Lemma \ref{maximalnilpotent}, $N$ is also the nilradical of $T$. Moreover, by Lemma \ref{heredecomp} and \ref{propertyPfiber}, $T\cap B$ is cocompact in $T$. Then it follows from Lemma \ref{Mostowstructure} that $N\cap B=N\cap (T\cap B)$ is cocompact in $N$ and hence has a finite invariant covolume.
    By Lemma \ref{Admissible iff finite volume}, $N\cap B$ is an admissible subgroup in $N$.
    
    Since $N\cdot B$ is closed by Lemma \ref{propertyPfiber}, 
    and $C_1N$ is a characteristic (hence normal) subgroup of $G$ by Lemma \ref{charsubgroup},
    we see that $(C_1N)\cdot B=C_1\cdot (N\cdot B)$ is a closed subgroup of $G$.
    Then by Lemma \ref{heredecomp} and \ref{propertyPfiber} again, $(C_1N)\cap B$ is cocompact in $C_1N$; and by Lemma \ref{Mostowstructure} and \ref{Admissible iff finite volume} again, 
    $(C_1N)\cap B$ is an admissible subgroup in $C_1N$.
\end{proof}

\section{Classification of orbits under twisted conjugacy}\label{app sec: twisted conjugacy}

The following key lemma is due to Jinpeng An, where it appears in his work addressing the classification of orbits under twisted conjugacy actions on a general connected Lie group. For more discussions see \cite{An-thesis, An-Geo}. 

\begin{lemma}\label{cleverajp}
Let $G$ be a connected Lie group, and let $A$ be an automorphism of $G$ such that $DA$ has no eigenvalue $1$. Then the map 
\[
\sigma\colon G\to G,\qquad \sigma(g)=gA(g)^{-1}
\]
is surjective.
\end{lemma}

\begin{proof}
    We prove by induction on $\dim{G}$. If $\dim{G}=0$, the lemma is obvious. Suppose that $\dim{G}=n>0$ and the lemma holds for connected Lie groups of dimension less than $n$. If $G$ is not solvable, then every automorphism of $\mathfrak{g}=\mathrm{Lie}(G)$ has $1$ as an eigenvalue.
    Hence we may assume that $G$ is solvable.  
Without loss of generality, we may also assume that $G$ is simply connected. Then the commutator subgroup $[G,G]$ is a closed, connected normal subgroup of $G$ with $\dim{[G,G]}<n$. Note that $A([G,G])=[G,G]$. By the inductive hypothesis, we have $\sigma([G,G])=[G,G]$.

Consider the connected abelian Lie group $Q:=G/[G,G]$ and the map
\[
\overline{\sigma}:Q\to Q,\qquad \overline{\sigma}([g])=[\sigma(g)]
\]
where $g\in G$ and $[g]=g[G,G]$. Since $Q$ is abelian, $\overline{\sigma}$
is a homomorphism. Moreover, since $DA$ has no eigenvalue $1$, $D\sigma=\mathrm{id}-DA$ is a linear isomorphism of $\mathfrak{g}$. It follows that 
$D\overline{\sigma}$ is an isomorphism of $\mathfrak{q}=\mathrm{Lie}(Q)$.
Therefore, $\overline{\sigma}$ is a surjective homomorphism.

Now for any $g\in G$, there exists an $h\in G$ such that $[g^{-1}]=\overline{\sigma}([h])=[\sigma(h)]$, or in other words
$g^{-1}=\sigma(h)g'$ for some $g'\in [G,G]$.
Since $h(g')^{-1}h^{-1}\in [G,G]$ and 
$hA(h)h^{-1}A(h)^{-1}\in [G,G]$,
we have 
$h(g')^{-1}A(h)h^{-1}A(h)^{-1}\in [G,G]$,
and so there exists an $h'\in [G,G]$
such that $\sigma(h')=h(g')^{-1}A(h)h^{-1}A(h)^{-1}$.
It follows that 
\[
g=(g')^{-1}\sigma(h)^{-1}
=(g')^{-1}A(h)h^{-1}
=h^{-1}\sigma(h')A(h)
=\sigma(h^{-1}h').
\]
In particular, $g$ is in the image of $\sigma$. This proves that $\sigma$ is surjective.
\end{proof}

\begin{proof}[Proof of Lemma \ref{eigen1fix}]
 By Lemma $\ref{cleverajp}$, there exists $g\in G$ such that 
 $a=gA(g)^{-1}$. So $g=aA(g)=f(g)$ is a desired fixed point.
\end{proof}

%


\begin{proof}[Proof of Lemma \ref{eigen1trans}]
    For $a\in G$, let $c_a$ denote the inner automorphism of $G$ induced by $a$: $c_a(g)=aga^{-1}$. Suppose that $A_1\circ A_2^{-1}=c_a$, and suppose to the contrary that $1$ is not an eigenvalue of $DA_2$. Then by Lemma $\ref{cleverajp}$, there exists $g\in G$ such that 
$a=gA_2(g)^{-1}$. Then 
\[
A_1=c_{gA_2(g)^{-1}}\circ A_2=c_g\circ A_2\circ c_{g^{-1}}.
\]
It follows that $1$ is not an eigenvalue of $DA_1=\mathrm{Ad}(g)\circ DA_2\circ \mathrm{Ad}(g)^{-1}$, a contradiction.
\end{proof}

\end{document}